\documentclass[a4paper,11pt]{amsart}
\usepackage{graphicx} % Required for inserting images
\usepackage{amsmath,amssymb,amsthm}
\usepackage{paralist}
\usepackage{quiver}
\usepackage{framed}

\let\oldtocsection=\tocsection
\let\oldtocsubsection=\tocsubsection
\let\oldtocsubsubsection=\tocsubsubsection
 
\renewcommand{\tocsection}[2]{\hspace{0em}\oldtocsection{#1}{#2}\bfseries}
\renewcommand{\tocsubsection}[2]{\hspace{1.8em}\oldtocsubsection{#1}{#2}}
\renewcommand{\tocsubsubsection}[2]{\hspace{4.4em}\oldtocsubsubsection{#1}{#2}}
\makeatletter
\renewcommand\subsection{\@startsection{subsection}{2}%
  \z@{-.5\linespacing\@plus-.7\linespacing}{.5\linespacing}%
  {\normalfont\scshape}}
\renewcommand\subsubsection{\@startsection{subsubsection}{3}%
  \z@{.5\linespacing\@plus.7\linespacing}{.5\linespacing}%
  {\normalfont\scshape}}

\makeatother
\usepackage[framemethod=tikz]{mdframed}

\usepackage[pdftex]{hyperref}
\hypersetup{
    colorlinks=true, %set true if you want colored links
    linktoc=all,     %set to all if you want both sections and subsections linked
    linkcolor=blue,  %choose some color if you want links to stand out
    allcolors=blue,
}

\newcommand{\bb}[1]{\mathbb{#1}}
\newcommand{\cl}[1]{\mathcal{#1}}
\newcommand{\fk}[1]{\mathfrak{#1}}

\newcommand{\deff}[1]{\emph{#1}}

\newcommand{\N}{\bb{N}}
\newcommand{\M}{\bb{M}}

\newcommand{\1}{1_{\B}}
\newcommand{\0}{0_{\B}}

\newcommand{\FPHP}{F^0_{\mathrm{PHP}}}
\newcommand{\GPHP}{G^0_{\mathrm{PHP}}}

\newcommand{\Ffull}{F_{\mathrm{PHP}}}
\newcommand{\Gfull}{G_{\mathrm{PHP}}}

\newcommand{\OPHP}{\Omega^0_{\mathrm{PHP}}}
\newcommand{\PHP}{\mathrm{PHP}}
\newcommand{\ontoPHP}{\mathrm{ontoPHP}}
\newcommand{\ontoFPHP}{\mathrm{ontoFPHP}}

\renewcommand{\Mc}{\cl{M}}
\newcommand{\A}{\cl{A}}
\newcommand{\Ic}{\cl{I}}
\newcommand{\B}{\cl{B}}
\newcommand{\RR}{\bb{R}}
\newcommand{\abs}[1]{\lvert #1\rvert}
\def\llbracket{{[\![}}
\def\rrbracket{{]\!]}}
\newcommand{\bbl}{\llbracket}
\newcommand{\bbr}{\rrbracket}
\newcommand{\Land}{\bigwedge}
\newcommand{\Lor}{\bigvee}

\newcommand{\forcingFrame}{(P, \preceq)}
\newcommand{\forcingAlgebra}{\cl{B}(\Vdash)}
\newcommand{\Th}{{\sf{Th}}}

\newcommand{\ZF}{\mathrm{ZF}}
\newcommand{\GB}{\mathrm{GB}}
\newcommand{\ACA}{\mathrm{ACA}_0}
\newcommand{\PA}{\mathrm{PA}}

\theoremstyle{definition}
\newtheorem{theorem}{Theorem}[section]
\newtheorem{definition}[theorem]{Definition}
\newtheorem{lemma}[theorem]{Lemma}
\newtheorem{corollary}[theorem]{Corollary}
\newtheorem{proposition}[theorem]{Proposition}
\newtheorem{remark}[theorem]{Remark}

\newtheorem{problem}[theorem]{Problem}

\newtheorem*{theorem*}{Theorem}

\newcommand{\bv}[1]{\llbracket{#1}\rrbracket}

\newcommand{\XiM}{F^{\mathrm{max}}}
\renewcommand{\Xi}{F}
\newcommand{\AU}{\A/U}
\newcommand{\Then}{\Longrightarrow}
\renewcommand{\A}{\mathcal{A}}

\newcommand{\Km}{\mathcal{M}_K^\Xi}
\newcommand{\Kmax}{\mathcal{M}_K^{\XiM}}
\newcommand{\UKm}{\mathcal{M}_K^\Xi/U}
\newcommand{\UKmG}{\mathcal{M}_K^\Xi/{U_g}}
\newcommand{\LUKmG}{(\mathcal{M}_K^\Xi/{U_g},<)}

\newcommand{\UKmPHP}{\mathcal{M}_K^{\FPHP}/{U_g}}

\newcommand{\UKmax}{\mathcal{M}_K^{\XiM}/U}
\newcommand{\UKmaxG}{\mathcal{M}_K^{\XiM}/{U_g}}
\newcommand{\LUKmaxG}{(\mathcal{M}_K^{\XiM}/{U_g},<)}

\renewcommand{\M}{\mathcal{M}}
\newcommand{\Mg}{M[G]^{\mathcal{R}}}

\newcommand{\UltMMax}{\mathrm{Ult}(\M,\XiM,U)}
\newcommand{\UltMM}{\mathrm{Ult}(\M,\Xi,U)}

\newcommand{\UltMMaxG}{\mathrm{Ult}(\M,\XiM,{U_g})}
\newcommand{\UltMMG}{\mathrm{Ult}(\M,\Xi,{U_g})}

\newcommand{\sub}{\subseteq}
\newcommand{\brm}{\begin{remark}\begin{rm}}
\newcommand{\erm}{\end{rm}\end{remark}}
\newcommand{\bce}{\begin{compactenum}}
\newcommand{\ece}{\end{compactenum}}

\newcommand{\set}[2]{\{#1 \: ; \: #2\}}
\newcommand{\seq}[2]{\langle#1 \: ; \: #2\rangle}

\newcommand{\AMno}{Atserias and M{\"u}ller}

\newcommand{\AMc}{Atserias and M{\"u}ller~\cite{atseriasmuller15} }
\newcommand{\Kr}{Krajíček }
\newcommand{\Krno}{Krajíček}

\newcommand{\Krcno}{Krajíček \cite{krajicek11}}

\newcommand{\Los}{Łoś }
\newcommand{\Bi}{\B_{\aleph_1}}

\newcommand{\Bo}{\B_{\omega}}
\newcommand{\idU}{[id]_{U_g}}
\newcommand{\aU}{[\alpha]_{U_g}}
\newcommand{\Bk}{\mathcal{B}_\kappa}

\renewcommand{\P}{\mathbb{P}}
\newcommand{\Co}{\mathcal{C}_{\aleph_0}}

\newcommand{\Pfin}{\tilde{\mathcal{P}}(I)/\mathrm{Fin}}
\newcommand{\Pomega}{\tilde{\mathcal{P}}(\omega)/\mathrm{Fin}}
\newcommand{\fin}{\mathrm{Fin}}

\newcommand{\CH}{\mathrm{CH}}
\newcommand{\CoRO}{\mathrm{RO}(\Co)}

\begin{document}

\title[Forcing in bounded arithmetic and in set theory \ldots]{Forcing in bounded arithmetic and set theory: analogies and differences}

\author{Radek Honzik}
\address[Honzik]{
Charles University, Faculty of Arts,
Celetn{\' a} 20, Prague~1, 
116 42, Czech Republic
}
\email{radek.honzik@ff.cuni.cz}
\urladdr{https://www1.cuni.cz/~honzikr/}

\author{Ondřej Ježil}
\address[Ježil]{
University of Warwick, Department of Computer Science
Coventry
CV4 7AL, United Kingdom
}
\email{ondrej.jezil@email.cz}
\urladdr{ojezil.github.io}

\author{Mykyta Narusevych} \address[Narusevych]{ Charles University, Faculty of Mathematics and Physics, Czechia } \email{mykyta.narusevych@gmail.com} \urladdr{}

\thanks{
All three authors were supported by Czech Science Foundation (GA{\v C}R) grant ``The role of set theory in modern mathematics'' (24-12141S). The second author received support from the UKRI Frontier Research Guarantee Grant EP/Y007999/1 and the Centre for Discrete Mathematics and its Applications (DIMAP) at the University of Warwick. The second author would also like to thank the Isaac Newton Institute for Mathematical Sciences, Cambridge, for support and hospitality during the programme Logical Foundations of Complexity Theory (LFC), where part of the work on this paper was undertaken. The LFC programme was supported by EPSRC grant EP/Z000580/1.}

\begin{abstract}
We provide a modern overview of several forcing methods in bounded arithmetic and compare them with forcing in set theory. We specifically discuss the countable forcing framework of Atserias and M{\"u}ller \cite{atseriasmuller15}, the restricted ultrapowers of Bydžovský and M{\"u}ller \cite{bydvzovsky2020polynomial}, and forcing with random variables developed by \Kr \cite{krajicek11}. We illustrate these forcing arguments by reviewing the Paris--Wilkie forcing and by giving a self-contained construction of the shallow PHP model from \cite{krajicek11}. We observe that, from a formal point of view, \Krno's forcing with uncountable structures can be interpreted in the framework of Atserias and M{\"u}ller after a countable collapse using a L{\"o}wenheim--Skolem argument. We further observe that the countable restricted ultrapowers from \cite{bydvzovsky2020polynomial}, used to obtain models of the theory $T_{PV}$, can also be interpreted in this framework. This part serves as a primer on modern forcing methods in bounded arithmetic. In the next part, we discuss the analogies and differences between forcing in arithmetic and set theory. Building on Honzik \cite{RH:bounded}, we embed \Krno's models into generic extensions of models of set theory. In our interpretation, set theory offers a convenient framework for formalizing forcing arguments with uncountable structures. We end with a question whether analogies with set theory can facilitate the discovery of new uncountable partial orders for arithmetic.\end{abstract}

\maketitle

\tableofcontents

\section{Introduction}\label{sec:1}

Since the discovery of forcing in set theory by Cohen \cite{COHch}, researchers have actively explored the possibilities for extending forcing to other theories. In this article we focus on the status of this project in bounded arithmetic---a collective name for weak subtheories of Peano Arithmetic in which the schema of induction is limited to formulas with (appropriately specified) bounded quantifiers. Our primary goal is to provide a unified exposition of the existing forcing methods in bounded arithmetic and critically analyze their analogies and differences with forcing in set theory; ultimately, we want to explore whether the extensive range of forcing arguments in set theory can provide new insights into forcing in arithmetic. 

As a starting point, we consider the framework of \emph{partially definable forcings} defined by Atserias and M{\"u}ller in \cite{atseriasmuller15} (we will denote this framework AM) and verify that it can be used to interpret another forcing framework developed by Kraj{\'i}{\v c}ek in \cite{krajicek11} using \emph{random variables} (framework K). To satisfy the formal constraints of framework AM, this interpretation of K involves taking a countable elementary submodel yielding a countable version of Kraj{\'i}{\v c}ek's forcing. Additionally, we observe that certain \emph{countable restricted ultrapowers} (framework U) provide another way to look at forcing in bounded arithmetic, also interpretable in framework AM.

The frameworks AM and U are naturally based on countable (nonstandard) models of arithmetic. While the framework K can also be collapsed into a countable setting, it was originally developed using uncountable $\aleph_1$-saturated models and uncountable complete Boolean algebras; the uncountable setting is more natural and intuitive. We observe that forcing with uncountable models is conveniently and elegantly interpreted using the standard notion of set-theoretic forcing and generic ultrafilters. In this interpretation, forcing K can be viewed as yielding uncountable restricted ultrapowers (in the arithmetic language) existing in generic extensions of transitive models of set theory. Moreover, in this perspective, it was much easier to observe that the measure algebra used in \cite{krajicek11} is in fact isomorphic to a standard random measure algebra associated with the non-separable product measure space $2^\kappa$ for an appropriate uncountable $\kappa$ (see the article \cite{RH:bounded} by the first author for more details).

Since all countable non-trivial partial orders are forcing-equivalent\footnote{We say that two partial orders are forcing-equivalent if they have isomorphic Boolean completions.} to the Cohen partial order composed of finite binary sequences ordered by the reverse inclusion (denoted $\Co$), these facts lead to an interesting observation that in bounded arithmetic only two types of partial orders, or equivalently, two types of complete Boolean algebras, have been used for forcing arguments: the countable Cohen forcing $\Co$ (Frameworks AM\footnote{From the external perspective mentioned in the bullet in the main text, it supports all forcing relations despite its countable setting. For instance, as we show in Section \ref{sec:33} and comment in Footnote \ref{ft:ext}, it can interpret \Krno's forcing $\Vdash^M_{\Bk}$ with the random algebra. However, it is not clear how to define $\Vdash^M_{\Bk}$ internally in the countable setting of AM.}, U) and the random algebras associated with the measure spaces $2^\kappa$ (denoted $\Bk$) (Framework K). This contrasts with the apparent richness of partial orders $\P$ used in set-theoretic forcing.

The richness of partial orders in set theory can viewed from two complementary perspectives:

\begin{itemize}

\item \emph{The internal perspective.} If $M$ is a transitive model of set theory, it contains many uncountable forcing non-equivalent partial orders $\P \in M$ (uncountable and non-equivalent in the sense of $M$). The forcing relations $\Vdash^M_\P$ corresponding to $\P$ are defined internally in $M$ and provide in $M$ a proper class of forcing notions $\Vdash^M_\P$.

\item \emph{The external perspective.} Externally in $V$, all non-trivial partial orders $\P$ in countable transitive models of set theory are equivalent to Cohen forcing $\Co$. However, since $\Vdash^M_\P$ are developed internally in $M$, the external countability of $\P$ does not play a role in their definition; consequently,  the countably many forcing relations $\Vdash^M_\P$ provide a range of forcing relations different from the Cohen forcing $\Vdash^M_{\Co}$.
\end{itemize}

In contrast, if $M$ is a model of arithmetic, uncountability cannot be expressed internally in $M$, which makes it more difficult to come up with forcing relations different from the Cohen forcing (Framework K is an exception to this rule). By embedding models of arithmetic in transitive models of set theory, this restriction seems less pressing and dealing with uncountable models becomes more natural (see Section \ref{sec:4} here and the article of the first author \cite{RH:bounded} where this point is developed in detail). In our interpretation, using the set-theoretic forcing framework may be helpful in finding new candidates for uncountable forcing notions applicable to uncountable models of arithmetic, with some of the methods and techniques in set theory being partly transferable to forcing in arithmetic (see Section \ref{sec:open} with open questions).
 
The article is structured as follows. 

In Section \ref{sec:2} we briefly review theories in bounded arithmetic relevant to forcing arguments in order to make the article self-contained.

In Section  \ref{sec:3} we discuss three frameworks for forcing in bounded arithmetic. In Section \ref{sec:31} we briefly review the framework proposed by \AMc for the specific use in bounded arithmetic, with an example of the Paris--Wilkie forcing. In Section \ref{sec:32} we give a stand-alone presentation of forcing with random variables from \Krcno, with the focus on presenting the construction of the shallow PHP model $K(\FPHP,\GPHP)$ originally from~\cite[Chapter 20]{krajicek11}.  Unlike in Krajíček's original book, we prove all the required results about this structure explicitly without referencing earlier parts of the book (this part can work as a primer to forcing with random variables). In Section~\ref{restricted} we discuss countable restricted ultrapowers and the results of Bydžovský and Müller on models of the theory $T_{PV}$~\cite{bydvzovsky2020polynomial}.
 
In Section \ref{sec:33} we observe that if the forcing with random variables from \cite{krajicek11} is collapsed into a countable system, it can be---from  the formal point of view---interpreted in the framework of \AMno.\footnote{\label{ft:ext} Thus providing a countable forcing relation $\Vdash^M_{\Bk}$ different from $\Vdash^M_{\Co}$ in the sense of the external perspective discussed above.} However, this (formal) interpretation does not provide additional insights into the forcing from \cite{krajicek11}. Moreover, we show that countable restricted ultrapowers can also be interpreted in the framework of \AMno. These results suggest that, at the moment at least, the framework developed in \AMc is sufficiently general to serve as a general forcing framework in arithmetic.

In Section \ref{sec:4} we look at forcing with random variables from the perspective of set-theoretic forcing and uncountable structures. We first make a few notes regarding the (faithful) collapse of Boolean-valued models to $\{0,1\}$-valued models, based on the notion of (restricted) fullness of the models, and generic ultrafilters (Sections \ref{sec:41} and \ref{sec:42}). As the first author showed in \cite{RH:bounded}, the complete Boolean algebras from \cite{krajicek11} are isomorphic to standard random measure algebras $\Bk$. This provides parallels between forcing with these algebras in set theory and in arithmetic which we discuss in some detail in Sections \ref{sec:set}.

In Section \ref{sec:open} we state some open questions.

\medskip

\textbf{Acknowledgement.} The authors would like to thank Moritz M{\"u}ller and Ján Pich for helpful comments regarding a draft of this paper.

\section{Forcing in bounded arithmetic}\label{sec:2}

\subsection{A short history} \label{sec:history}

Bounded arithmetic is a collective name for weak subtheories of Peano Arithmetic~\cite{buss85,krajicek95,cook2010logical}. These theories have close connections to proof complexity of propositional proof systems, that is to the task of establishing bounds on the size of any proof of a propositional tautology and each such theory is in a precise technical sense related to a specific complexity class. For example lower bounds on the proof size in the constant depth Frege system originally trace to a forcing construction in bounded arithmetic by Ajtai~\cite{ajtai1994complexity}, and state-of-the-art Boolean circuit size lower bounds~\cite{chen2024symmetric,li2024symmetric} trace back to provability in bounded arithmetic~\cite{jevrabek2004dual}. See the monograph~\cite{cook2010logical} for systematical presentation of this connection. 

Historically, first use of forcing in bounded arithmetic was by Paris and Wilkie in~\cite{paris2006counting}, followed by its use in Ajtai's theorem~\cite{ajtai1994complexity}. The result of Paris and Wilkie has been subsequently generalized by Riis~\cite{riis94}, which was in turn generalized by Müller~\cite{muller20} as an instantiation of the framework of Atserias and Müller \cite{atseriasmuller15}. Another use of this framework was by Narusevych~\cite{narusevych24} to extend results of Müller.

A version of Boolean-valued forcing in bounded arithmetic was considered by Takeuti and Yasumoto~\cite{takeuti1996forcing,takeuti1998forcing}. Kuroda provided an overview of their approach in~\cite{kuroda2018developing} and showed that an earlier conditional forcing construction of Krajíček~\cite{krajicek95onfrege} can be cast in the framework of Takeuti and Yasumoto. An extension of Ajtai's theorem has been proved by Zambella~\cite{zambella1996forcing} using a forcing construction which translates combinatorial statements used by Ajtai into direct model-theoretic arguments. 

In \cite{krajicek11}, Krajíček presented a well-developed Boolean-valued framework for constructing models of both one-sorted and two-sorted arithmetic; one particular result in \cite{krajicek11} is the construction of the model $K(\Ffull,\Gfull)$ which implies a stronger version of Ajtai's theorem. The framework of Krajíček has been cast as a way of expanding a limit of finite structures into a model of two sorted arithmetic by the second author~\cite{jezil2026limits}.

\subsection{Theories}\label{sec:prelim}

Let us briefly review some important theories of bounded arithmetic which motivate our exposition. For a more general context, additional theories and their relationships, the reader can consult the monographs~\cite{cook2010logical, krajicek95} and the dissertation thesis~\cite{buss85}, where a large portion of the modern systems of bounded arithmetic were introduced. Sections about bounded arithmetic can also be found in \cite[Chapter V]{hajekpudlak1995} and~\cite[Section 6.2]{ajtai1994complexity}.

\subsubsection{The theory $I\Delta_0$ and the Macintyre's problem}
The first theory of bounded arithmetic $I\Delta_0$ was introduced by Parikh~\cite{parikh1971existence} as a simple restriction of Peano Arithmetic where we only accept the induction axiom 
\[(\varphi(0)\land (\forall x)(\varphi(x)\to \varphi(x+1)))\to (\forall x)(\varphi(x)),\]
if $\varphi$ is a bounded formula. That is, $\varphi$ is for instance of the form
\[(\forall x_1\leq t_1)(\exists x_2\leq t_2)\dots (\forall x_{k-1}\leq t_{k-1})(\exists x_k\leq t_k)(\varphi_0(x_1,\dots,x_k,y_1,\dots,y_l)),\]
for some terms $t_1,\dots, t_k$, where none of the terms $t_i$ contains $x_i$ as a variable. Known unprovability results for $I\Delta_0$ are limited to Gödel's incompleteness theorem and consequences of what is now known as Parikh's theorem which essentially asserts that functions of superpolynomial growth rate are not provably total in $I\Delta_0$. Neither of these approaches can be used in to show unprovability of universal closures of bounded formulas in suitable generality. The $\Delta_0$-pigeonhole principle scheme, which can be written as a countable collection of universal closures of bounded formulas, states that no bounded formula defines the graph of an injection of $\{0,\dots,a+1\}$ to $\{0,\dots,a\}$. It is still unknown whether $I\Delta_0$ proves this scheme (the question is known as the \emph{Macintyre's problem}). 

\subsubsection{Relativized, Two-sorted and One-sorted Theories}

In this section, we describe different variants of bounded arithmetic theories and tie their constructions back to Macintyre's problem, which we use as an anchor point. A non-trivial unprovability result can be obtained if the basic language of Peano Arithmetic, $L_{PA}=\{0,1,+,\cdot,\leq\},$ is extended by a new binary relational symbol $R$ that can appear in the induction formulas, but there are no other restrictions for its interpretation.  The resulting theory, denoted  $I\Delta_0(R)$, can prove some formulas related to the relation $R$. For example, since $R(x,y)$ satisfies induction, it also has to satisfy the least number principle in models of $I\Delta_0(R)$. A theory which is expanded in this way is called a \emph{relativized theory}.

By a well-known result of Ajtai~\cite{ajtai1994complexity}, it follows that Macintyre's problem for $I\Delta_0(R)$ has a negative solution. That is, $I\Delta_0(R)$ cannot prove that for every number $a$ the relation $R$ does not represent the graph of an injective function from $\{0,\dots,a+1\}$ to $\{0,\dots,a\}$. It follows there is a model $(M,R)$ of $I\Delta_0(R)$, where $R$ is a graph of function from $\{0,\dots,n+1\}$ to $\{0,\dots,n\}$ for some non-standard number $n\in M$.

If $n$ is non-standard, it is easy to find the graph of an injection from  $n+1$ to $n$ since both the sets $\{0,\dots,n+1\}$ and $\{0,\dots,n\}$ are of the same (infinite) cardinality. The hardness stems from the fact that we need the relation $R$ to satisfy bounded induction. To construct such a relation for any suitably ``short'' model $M$ of $I\Delta_0$, Ajtai employs a version of Robinson's model-theoretic forcing combined with an involved combinatorial argument (a switching lemma).

Let us now postpone the discussion of how relativized theories relate to the original Macintyre's problem, and let us introduce another theory. The theory $I\Delta_0(R)$ proves statements about numbers and the fixed binary relation $R$. We can capture this duality more explicitly by a  two-sorted theory $V^0_1$, which can quantify both numbers and sets. Its axiomatization starts with 15 axioms forming a theory denoted $2$-\textbf{BASIC} that gives basic properties of the function symbols in $L_{PA}=\{0,1,+,\cdot,\leq\}$ and of the set variables denoted by capital letters $X,Y,\dots$ that represent bounded sets of numbers. Additionally, there is the length function $\abs{X}$ defined to map the set sort to the number sort and output the succcessor of the largest number in $X$. This set of axioms is extended by the $\Sigma^B_0$-\textbf{COMP} scheme that asserts for every formula $\varphi$ without set quantifiers, with every number quantifier bounded, the following formula
\[(\exists Y\leq y)(\forall z<y)(X(z)\leftrightarrow \varphi(z)),\]
is an axiom of $V^0_1$ (where $X(z)$ means that $z$ is a member of $X$ and the bounded set quantifier $(\exists Y\leq y)(\dots)$ is a short for $(\exists Y)(\abs{Y}\leq y\land \cdots)$).\footnote{For more details about the definition of $V^0_1$ and its extensions $V^i_1$, $i\geq 1$, see~\cite[Section V.1]{cook2010logical}}

The theory $V^0_1$ proves induction for formulas with bounded number quantifiers and moreover it holds that
\begin{equation} \label{exp} V^0_1 \vdash ( \forall X)\varphi(X) \iff I\Delta_0(R)\vdash \varphi(R),\end{equation}
where $\varphi(X)$ is a formula with all quantification bounded over the number sort and with a free set variable $X$ and $\varphi(R)$ is the resulting formula obtained by replacing every occurence of the elementhood predicate\footnote{The theory $I\Delta_0$ can prove the basic properties of the pairing function and by $(x,y)$ we simply denote the number coding the pair.} $X((x,y))$ by $R(x,y)$. Furthermore, a model $(M,R)\models I\Delta_0(R)$ can be expanded to a model of $V^0_1$ by introducing a new set for every predicate definable using a bounded $L_{PA}(R)$-formula.\footnote{Analogous constructions can be carried out with theories $\ZF$ and $\GB$, or $\PA$ and $\ACA$. See \cite[Chapter III(b)]{hajekpudlak1995} for comparison.}

One might be worried whether provability in two-sorted arithmetic theories still has some relation with provability in the unrelativized one-sorted arithmetic, let alone if there is any hope of solving the Macintyre's problem for $I\Delta_0$ this way. Fortunately a construction known as the \emph{$RSUV$-isomorphism}~\cite{takeuti1993rsuv} (see also~\cite[Section 5.5]{krajicek95}) lets us take a model of a two-sorted theory and transform it to a model of a one-sorted theory, by mapping (i) the bounded sets to numbers whose binary expansions code the characteristic function of the set, and (ii) the numbers to their \emph{lengths}, where lengths are identified with numbers $x$ for which $2^x$ provably exists in the theory (the length of a number is thus logarithmic to its magnitude). The strength of the two-sorted theory we start with determines the strength of the one-sorted theory we end up with. Unfortunately, to get a theory which is at least as strong as $I\Delta_0$, we need to start with a model of $V_1$, a~theory which extends $V^0_1$ by allowing both bounded number quantification and bounded set quantification in formulas used in its comprehension scheme. However, forcing a new bounded set $X$ into existence so that it preserves even one bounded existential quantifier for sets in the induction formulas seems currently out of reach. The progress has actually stalled at preserving induction for quantifiers counting elements of the number sort modulo $2$ (the corresponding theory is sometimes denoted $Q_2V^0_1$ and it is known that it is much weaker than $V_1$, see~\cite[Chapters 13 and 22]{krajicek11} for details).

In Figure~\ref{fig:diagramth} we give a diagram depicting some important theories of bounded arithmetic. Currently, the only known forcing construction (in the sense of this article) is known for $I\Delta_0(R)$, and thus for $V^0_1$.  Already for $Q_2V^0_1$, which allows quantifiers counting modulo $2$ in its comprehension scheme, there are no known forcing constructions (and similarly for the extensions of $V^0_1$ by faster growing functions $V^0_k$,  $k\in\mathbb{N}$, which we omit from the diagram). The next theory $V^1_1$ allows bounded existential quantification over the set sort in its comprehension scheme. The theory $V^1_1$ is $RSUV$-isomorphic to the well known theory $S^1_2$ which is often described as a theory formalizing reasoning with polynomial-time concepts. Separating $S^1_2$ from $S_2$ (or, equivalently, $V^1_1$ from $V_1$) is considered one of the main problems of bounded arithmetic. Showing unprovability of a bounded two-sorted formula in $V^1$ translates by the $RSUV$-isomorphism to the unprovability of a bounded one-sorted formula in $S_2$. Finally, $S_2$ is conservative over an extension of $I\Delta_0$ by the totality of the function $x^{\log(x)}$.~\footnote{For details regarding the theories $S^i_2$,$i\geq 1$; $S_2$ and their relationship to $I\Delta_0$ see~\cite[Chapter 5]{krajicek95}.}

\begin{figure}
    \centering
    \[\fbox{
    \begin{tikzcd}[ampersand replacement=\&]
	{I\Delta_0(R)} \&\&\&\&\& \\
	{V^0_1} \& {Q_2V^0_1} \& {V^1_1} \& {V^2_1} \& \dots \& {V_1} \\
	\&\& {S^1_2} \& {S^2_2} \& \dots \& {S_2} \\
	\&\&\&\&\& {I\Delta_0}
	\arrow[from=1-1, to=2-1]
	\arrow["\leq"{description}, draw=none, from=2-1, to=2-2]
	\arrow["\leq"{description}, draw=none, from=2-2, to=2-3]
	\arrow["\leq"{description}, draw=none, from=2-3, to=2-4]
	\arrow["{\scriptscriptstyle\underset{RSUV}{\cong}}"{marking, allow upside down}, draw=none, from=2-3, to=3-3]
	\arrow["\leq"{description}, draw=none, from=2-4, to=2-5]
	\arrow["{\scriptscriptstyle\underset{\tiny RSUV}{\cong}}"{marking, allow upside down}, draw=none, from=2-4, to=3-4]
	\arrow["\leq"{description}, draw=none, from=2-5, to=2-6]
	\arrow["{\scriptscriptstyle\underset{\tiny RSUV}{\cong}}"{marking, allow upside down}, draw=none, from=2-6, to=3-6]
	\arrow["\leq"{description}, draw=none, from=3-3, to=3-4]
	\arrow["\leq"{description}, draw=none, from=3-4, to=3-5]
	\arrow["\leq"{description}, draw=none, from=3-5, to=3-6]
	\arrow["\leq"{marking, allow upside down}, draw=none, from=4-6, to=3-6]
\end{tikzcd}
    }\]

    \caption{\small A diagram of relevant theories.  The arrow from $I\Delta_0(R)$ to $V^0_1$ denotes the expansion in the sense of (\ref{exp}). The relation $S \leq T$ says that $T$ is a stronger theory than $S$ in the sense that for every formula $\varphi$ in the language of $S$ which is also a formula in the language of $T$ if $S\vdash \varphi$ then $T \vdash \varphi$. The relation $V{\scriptscriptstyle\underset{RSUV}{\cong}}S$ means that $V$ can be transformed into $S$ and vice-versa using the $RSUV$-isomorphism.}
    \label{fig:diagramth}
\end{figure}
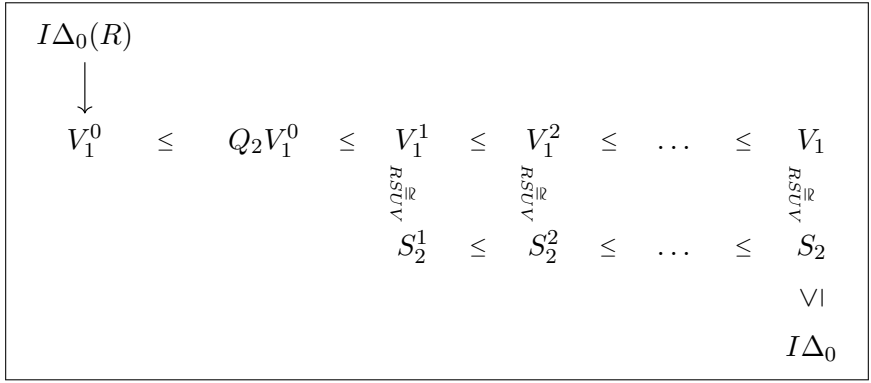

\section{Frameworks for the arithmetical forcing}\label{sec:two}\label{sec:3}

In this section we survey three important forcing frameworks developed for bounded arithmetic. The first framework is an extension of Abraham Robinson's model-theoretic forcing developed by  Atserias and  Müller in~\cite{atseriasmuller15}, which is presented in Section~\ref{sec:partially definable forcing} (framework AM). In Section~\ref{sectionkrajicek} we discuss the framework based on Boolean-valued models of arithmetic, as developed by Krajíček in~\cite{krajicek11} (framework K). In Section \ref{restricted} we discuss countable restricted ultrapowers (framework U).

The most difficult step in the proofs involving arithmetical forcing  is to show that the induction for certain classes of formulas is preserved by the forcing: the details of this task in a given scenario can be quite challenging, and trace to one of the combinatorial statements known as \emph{switching lemmas}.\footnote{\label{ft:sw}The switching lemmas are certain combinatorial lemmas used in proof complexity, circuit complexity and the related forcing arguments which describe simplification of a computational model under a random restriction on its input bits. The switching lemma for PHP-trees was developed in~\cite{krajivcek2008exponential,pitassi1993exponential} to improve upon Ajtai's lower bound. The PHP-tree switching lemma is analogous to an earlier result of Håstad~\cite{hastad1986almost} concerning to Boolean circuits and which itself is related to earlier results of Subbotovskaya~\cite{subbotovskaya1961realization}, Furst, Saxe and Sipser~\cite{furst1984parity}, Ajtai~\cite{ajtai198311} and Yao~\cite{yao1985separating}.  For more details, see an overview of Beame~\cite{beame1994switching} and of Thapen~\cite{thapen2022notes}.} While this part of the argument is often understood as the heart of the matter in the technical sense, we feel that it can be omitted for the expository purposes for the comparison of the forcing methods in bounded arithmetic and set theory. 

\subsection{Framework AM: Partially definable forcing} \label{sec:31}
\label{sec:partially definable forcing}
In this section we briefly review the framework proposed by Atserias and M{\"u}ller for forcing in bounded arithmetics \cite{atseriasmuller15}. We focus on properties which we will later use for comparison with other forcing frameworks.

\subsubsection{Definitions and main properties}

Throughout this section, we fix a countable model $M$ over a countable language $L$, a countable extension $L^*$ of $L$, and a non-empty countable partially ordered set $(P, \preceq)$ called a \deff{forcing frame}\footnote{The definition of forcing frame in \cite[Section 2.2]{atseriasmuller15} also includes a countable collection $\cl{D}$ of dense subsets of $P$ called \textit{padding} as a part of the frame. The padding does not influence the forcing relation $\Vdash$, but does affect the notion of \textit{genericity} because it determines which dense open sets should a generic filter meet.}. Elements of $P$ are called \deff{conditions}, and the fact that $p \preceq q$ is expressed as $p$ being an \deff{extension of} $q$. The language $L^*(M)$ ($L^*$ enriched by names for all the elements of $M$) is called a \deff{forcing language}. When referring to formulas, sentences, terms, etc., we implicitly fix the underlying language as $L^*(M)$, unless stated otherwise.

In the forcing context, one can imagine conditions being some partial specifications of a would-be structure. If $P$ does not contain $\preceq$-minima (which is the case in all the applications), then no condition specifies such a structure completely. Through these lenses, if two conditions $p, q$ lack a common extension, they are thought to specify contradictory information, reflected in calling them \deff{incompatible} (denoted as $p \perp q$). Compatibility of $p$ and $q$ is denoted as $p \| q$.

The approach to forcing taken in \cite{atseriasmuller15} is to start with a syntactically defined relation $\Vdash$ between conditions and sentences. The fact that $p \Vdash \varphi$ is expressed as $p$ \deff{forces} $\varphi$. The set of all conditions forcing a sentence $\varphi$ is denoted as $[\varphi]$.

What makes an arbitrary relation $\Vdash$ a forcing is a number of useful recurrences and topological properties of sets of the form $[\varphi]$.

\begin{definition}[{\cite[Definition 2.5]{atseriasmuller15}}]
\label{forcing}
    A binary relation $\Vdash$ between elements of $P$ and $L^*(M)$-sentences is called a \deff{forcing} if it satisfies the following properties for all sentences $\varphi, \psi$, all formulas $\chi(x)$, and all conditions $p, q$.
    \begin{enumerate}
        \item \deff{Universal pre-forcing recurrence}\footnote{There is a dual notion of \textit{existential pre-forcing recurrence} (\cite[Definition 2.3]{atseriasmuller15}) which, roughly speaking, takes $\exists, \lor, \neg$ as primitive connectives and expresses $\land, \forall$ via the classical dualities.}:
        \begin{itemize}
            \item $p \Vdash \neg \varphi \:$ if and only if  $\: \forall q \preceq p: \: q \nVdash \varphi$;
            \item $p \Vdash \varphi \land \psi \:$ if and only if  $p \Vdash \varphi$ and $p \Vdash \psi$;
            \item $p \Vdash \varphi \lor \psi \:$ if and only if $\: p \Vdash \neg(\neg \varphi \land  \neg \psi)$;
            \item $p \Vdash \forall x \chi(x) \:$ if and only if, for all $a \in M$, $\: p \Vdash \chi(a)$;
            \item $p \Vdash \exists x \chi(x) \:$ if and only if $\: p \Vdash \neg \forall x \neg \chi(x)$.
        \end{itemize}

        \item \deff{Extension}: if $p \Vdash \varphi$ and $q \preceq p$, then $q \Vdash \varphi$; i.e. the set $[\varphi]$ is \deff{downwards closed}.

        \item \deff{Stability}: if $\forall q \preceq p \: \exists r \preceq q \: : \: r \Vdash \varphi$, then $p \Vdash \varphi$.
    \end{enumerate}
\end{definition}

Already the universal pre-forcing recurrence is enough to derive several other useful properties of $\Vdash$, e.g., that the intersection $[\varphi] \cap [\neg \varphi]$ is empty. So as not to clutter the exposition, we are not listing them all, and instead will state any property with the corresponding reference only when we explicitly need it.

To build up \textit{semantics} tailored to a specific forcing relation $\Vdash$, we turn our attention to \textit{generic filters}. Recall that a \deff{filter} $G$ is a set of conditions satisfying 
\begin{itemize}
    \item for any $p, q \in G$, there exists $r \in G$ so that $r \preceq p, q$ (in particular, any two members of $G$ are compatible);
    \item for any $p \in G$, the filter $G$ contains all the conditions $q$ so that $p \preceq q$ (i.e., $G$ is \deff{upwards closed}).
\end{itemize}
A set $D \subseteq P$ is called \deff{dense} if any condition $p \in P$ can be extended to a condition $q \in D$. Then, \textit{generic filters} are the ones that intersect \textit{sufficiently many} dense sets (to be appropriately defined below). Note that filters cannot in general  intersect all dense sets, but as long as one is concerned with a countable family, such filters do exist (Rasiowa--Sikorsky lemma; essentially \cite[Lemma 2.12]{atseriasmuller15}).

\begin{definition}[{\cite[Definition 2.9]{atseriasmuller15} and also \cite{stern75}}]
\label{forcing algebra}
    The pair $(P, M)$ denotes a two-sorted structure, with the first sort carrying the frame $\forcingFrame$, and the second sort carrying the $L$-structure $M$. 
    
    For each atomic $L^*$-formula $\varphi(\overline{x})$ of arity $r$, we let $R_\varphi$ be a relation symbol of arity $r + 1$ of sort $P \times M^r$. We denote by $(P, M)^\Vdash$  the expansion of $(P, M)$ which interprets each $R_\varphi$ as the relation $\{(p, \overline{a}) \: ; \: p \Vdash \varphi(\overline{a})\}$.

    We define the \deff{forcing algebra} $\forcingAlgebra$ as the set of all subsets of $P$ definable in $(P, M)^\Vdash$. Finally, we say that a filter $G$ is \deff{generic}, if it intersects all dense subsets of $P$ which are  in $\forcingAlgebra$.\footnote{\label{ft:typ} In the \textit{typical} case, $P$ is a non-definable subset of $M$, $\preceq$ is definable in $M$, and all $R_\varphi$ are definable in $(M, P)$---an expansion of $M$ by a unary predicate defining $P$. Then, $\forcingAlgebra$ simply consists of subsets of $P$ definable in $(M, P)$.}
\end{definition}

For a fixed generic $G$, we consider the theory $\Th(G)$ of all sentences which are forced by some condition from $G$. It is then an easy exercise to show that $\Th(G)$ is complete and \textit{formally consistent} (i.e., $\varphi \in \Th(G)$, iff $\neg \varphi \notin \Th(G)$), and so one aims towards defining a suitable $L^*(M)$-structure $M[G]$ satisfying $\Th(G)$. To do so, let $\fk{T}(\Th(G))$ denote the \textit{Herbrand term structure}, which is an $L^*(M)$-structure on the universe of all closed terms, interpreting functions, relations and constants from $L^*(M)$ in the obvious way. We then consider a binary relation $\sim_{\Th(G)} := \{(s, t) \: ; \: s = t \in \Th(G)\}$. If $\sim_{\Th(G)}$ is a \textit{congruence} and all closed terms are $\sim_{\Th(G)}$-congruent to a constant from $M$, then we define the structure $M[G]$ as $\fk{T}(\Th(G))/\sim_{\Th(G)}$ and call it a \deff{generic associate} of $M$.\footnote{Note that it makes sense to define $\fk{T}(\Th(G))/\sim_{\Th(G)}$ even when closed terms are not necessarily congruent to constants from $M$. The latter property, however, is used in the proof of the crucial \textit{truth lemma} \cite[Theorem 2.19]{atseriasmuller15}.} It is clear that the universe of $M[G]$ can be described as $\{[a]_{\sim_{\Th(G)}} \: ; \: a \in M\}$; in other words, members of $M$ ``name" all the elements of $M[G]$.

All the above discussions lead to the following and crucial theorem relating the ``syntactical" forcing relation $\Vdash$ to its ``semantical" counterpart $M[G]$.

\begin{theorem}
\label{truth and completeness lemma}
    Assume $\Vdash$ is such that $M[G]$ is defined for every generic filter $G$. Let $\varphi$ be an arbitrary sentence and $p$ an arbitrary condition.
    \begin{enumerate}
        \item \deff{Truth Lemma}: $M[G]$ satisfies $\Th(G)$.

        \item \deff{Forcing Completeness}: $$p \Vdash \varphi \iff M[G] \vDash \varphi,$$ for all generic filters $G$ containing $p$.
    \end{enumerate}
\end{theorem}

Let us now address several aspects of this forcing framework, relevant for our exposition.

\subsubsection{Principal Theorem and the definability of the forcing relation}

One of the motivations behind framework AM was to formalize correspondence between the definability of the forcing relation $\Vdash$ in the ground model $M$ of arithmetic and the class of formulas for which the $<$-least number principle (induction) holds in generic extensions of $M$. Atserias and M{\"u}ller managed to show this correspondence for an important class of forcing notions which they called \emph{conservative}. Generic extension of conservative forcing notions are actually \emph{expansions}:   models $M[G]$ which have the same domain as $M$ (no new integers are added), but they have new relations. 

\brm
We note that frameworks K and U, which we discuss later, are non-conservative, hence the definability result of AM does not apply to them.
\erm

The following definitions and results are from \cite{atseriasmuller15}:
\begin{definition}
\label{forcing expansion and conservative forcing}
    Let $M[G]$ be a generic associate structure. We say $M[G]$ is a \deff{forcing expansion} if the map $a \mapsto [a]_{\sim_{\Th(G)}}$ is an isomorphism between $L(M)$-structures $M$ and $M[G]$.   Let $\Vdash$ be forcing. We say $\Vdash$ is \deff{conservative} if, for every condition $p$ and every atomic $L(M)$-sentence $\varphi$, $p$ forces $\varphi$ if and only if $\varphi$ is true in $M$ \footnote{While not mentioned in the paper \cite{atseriasmuller15} explicitly, it is not hard to show that a conservative forcing $\Vdash$ satisfies that every $p$ forces \textit{any} $L(M)$-sentence $\varphi$ which is true in $M$.}.
\end{definition}

\begin{theorem}[{\cite[Proposition 2.26]{atseriasmuller15}}]
\label{conservative forcing leads to forcing expansion}
    Let $\Vdash$ be forcing such that $M[G]$ is defined for every generic $G$. Then, $\Vdash$ is conservative if and only if each $M[G]$ is a forcing expansion.
\end{theorem}

A particular motivation behind developing such a forcing framework is to provide a general method for the construction of expansions of models of true arithmetic to models violating certain true arithmetic principles while at the same time preserving the \textit{least number principle} for as large a formula class as possible. To this end, \cite[Section 3]{atseriasmuller15} develops the so-called \textit{Principal Theorems}, providing certain definability conditions on conservative forcings for which the models $M[G]$ satisfy the mentioned principle.

Through the rest of this section, we assume $L$ contains a binary relation symbol $<$, and $M$ interprets $<$ as a strict linear ordering. We assume $M$ satisfies the least number principle with respect to $<$ for all $L(M)$-formulas.

For a term $t$ and a formula $\varphi$, notations $\forall x < t \: \varphi$ and $\exists x < t \: \varphi$ are shorthand for $\forall x (x < t \to \varphi)$, resp. $\exists x (x < t \land \varphi)$.

\begin{definition}
\label{partially definable forcing}
    Let $\varphi$ be a sentence and $p$ be a condition. Notation $p \| \varphi$ stands for $p \nVdash \neg \varphi$. Assume $\Vdash$ is a conservative forcing and $\varphi(\overline{x})$ is a formula. We say $\Vdash$ is \deff{definable for $\varphi(\overline{x})$} if, for any condition $p$, the set $\{ \overline{a} \: ; \: p \| \overline{a}\}$ is definable in $M$.
\end{definition}

\begin{theorem}[Principal Theorem {\cite[Theorem 3.5]{atseriasmuller15}}]
\label{principal theorem}
    Let $\Phi$ be a class of formulas and $\Vdash$ be a conservative forcing. If $\Vdash$ is definable for all formulas from $\Phi$, then $M[G]$ satisfies the $<$-least number principle for $\Phi$ for any generic $G$ for which $M[G]$ is defined.
\end{theorem}

A convenient general fact is that the class of formulas $\Phi$ for which $\Vdash$ is definable is \textit{closed under the existential quantification}. This yields the following form of the Principal Theorem.

\begin{corollary}
\label{strong principal theorem}
    Let $\Phi$ be a class of formulas and $\Vdash$ a conservative forcing. If $\Vdash$ is definable for all formulas from $\Phi$, then $M[G]$ satisfies the $<$-least number principle for all formulas $\exists \overline{x} \: \varphi(x, \overline{x})$, where $\varphi(x, \overline{x}) \in \Phi$, for any generic $G$ for which $M[G]$ is defined.
\end{corollary}

\begin{remark}
    In \cite[Definition 3.3]{atseriasmuller15}, a weaker notion of \textit{dense definability of $\Vdash$ up to a parameter $b \in M$} is considered. Principal Theorem \cite[Theorem 3.5]{atseriasmuller15} still applies, and if $\Vdash$ is (densely) definable up to $b$, then $M[G]$ satisfies the corresponding least number principle only up to $b$.

    Dense definability is instrumental in carrying the \textit{Ajtai forcing} \cite[Section 4.3]{atseriasmuller15}, while a more straightforward application (\cite[Section 4.1]{atseriasmuller15}) allows us to establish already the full definability of $\Vdash$.
\end{remark}

\subsubsection{An example: Paris-Wilkie forcing}

In this section, we exemplify the usage of Principal Theorem \ref{strong principal theorem} following~\cite[Section 4.1]{atseriasmuller15} which itself is a presentation of~\cite[Theorem 21]{paris2006counting}. Let $L$ be the language of \textit{Peano Arithmetic}, containing unary constants $0, 1$, binary functions $+, \times$, and binary relation $<$. Let $M$ be a countably infinite \textit{nonstandard model of true arithmetic}, and let $n$ be a \textit{nonstandard number}. Finally, $L^*$ extends $L$ by a single binary relation symbol $R$.

Let $\neg \ontoPHP(n, R)$ be sentence stating that ``$R$ is a perfect matching of the complete bipartite graph $K_{n+1,n}$", formally expressed as
\begin{gather*}
    \forall p < n + 1 \: \exists h < n : R(p, h)
    \\
    \land
    \\
    \forall h < n \: \exists p < n + 1 : R(p, h)
    \\
    \land
    \\
    \forall p < n + 1 \: \forall h, h' < n : R(p, h) \land R(p, h') \to h = h'
    \\
    \land
    \\
    \forall p, p' < n + 1 \: \forall h < n : R(p, h) \land R(p', h) \to p = p'.
\end{gather*}

Our goal is to expand $M$ by $R$ so that $\neg \ontoPHP(n, R)$ holds, and, at the same time, the least number principle is preserved for as large a formula class as possible. To this end, we define a frame $(P, \preceq)$, with conditions being finite matching of $K_{n+1, n}$ and $\preceq$ being a reverse inclusion. We now define a forcing relation $\Vdash$. Since we are aiming at universal conservative forcing (\ref{forcing}, \ref{forcing expansion and conservative forcing}), it is enough to specify it for sentences of the form $R(s, t)$, where $s, t$ are closed $L(M)$-terms. We let
\begin{align*}
    p \Vdash R(s, t) \iff \{s^M, t^M\} \in p,
\end{align*}
where $s^M, t^M$ denote values of the terms $s, t$ in $M$.

\begin{lemma}[{\cite[Section 4.1]{atseriasmuller15}}]
    The $\Vdash$ relation is a conservative forcing, and $M[G]$ is defined for any generic filter $G$. Furthermore, any condition $p$ (or, equivalently, the $\preceq$-largest condition $\varnothing$) forces $\neg \ontoPHP(n, R)$.
\end{lemma}

\begin{theorem}
\label{definability of paris-wilkie forcing}
    The forcing $\Vdash$ is definable for quantifier-free formulas.
\end{theorem}
\begin{proof} The original proof in \cite[Lemma 4.1]{atseriasmuller15} uses a more sophisticated method of \textit{definable antichains}. While not necessary for the purposes of our paper, such method is crucial for establishing definability of the \textit{Ajtai forcing}.

    Let $p$ be a condition, and $\varphi(\overline{x})$ be an open formula. We assume $\varphi(\overline{x})$ is of the form
    \begin{align*}
        \bigvee_{i < k} \varphi_i(\overline{x}) \land \varphi_i^+(\overline{x}) \land \varphi_i^-(\overline{x}), 
    \end{align*}
    where $k$ is finite, $\varphi(\overline{x})$ is a quantifier free $L(M)$-formula, $\varphi^+_i(\overline{x})$ is a conjunction of atomic formulas of the form $R(s(\overline{x}), t(\overline{x}))$, and $\varphi^-_i(\overline{x})$ is a conjunction of atomic formulas of the form $\neg R(s(\overline{x}), t(\overline{x}))$.

    Since $\Vdash$ satisfies universal forcing recurrence \ref{forcing}, it follows that $p \| (\psi \lor \theta)$ if and only if $p \| \psi$ or $p \| \theta$, for arbitrary sentences $\psi, \theta$. This implies
    \begin{align*}
        \{\overline{a} \: ; \: p \| \varphi(\overline{a})\} = \bigcup_{i < k} \{\overline{a} \: ; \: p \| \varphi_i(\overline{a}) \land \varphi_i^+(\overline{a}) \land \varphi_i^-(\overline{a})\}.
    \end{align*}
    
    Let $A_i$ be definable set of tuples $\overline{a}$ satisfying $\varphi_i(\overline{a})$ in $M$. It is then clear
    \begin{align*}
        \{\overline{a} \: ; \: p \| \varphi_i(\overline{a}) \land \varphi_i^+(\overline{a}) \land \varphi_i^-(\overline{a})\} = A_i \cap \{\overline{a} \: ; \: p \| \varphi_i^+(\overline{a}) \land \varphi_i^-(\overline{a})\}.
    \end{align*}

    For an atomic formula $R(s(\overline{x}), t(\overline{x}))$, $p \| R(s(\overline{a}), t(\overline{a}))$ if and only if $p \cup \{\{s^M(\overline{a}), t^M(\overline{a})\}\}$ is a matching of $K_{n+1, n}$, and, for a finite conjunction of such formulas $\bigwedge_{j < r} R(s_j(\overline{x}), t_j(\overline{x}))$, $p \| \bigwedge_{j < r} R(s_j(\overline{a}), t_j(\overline{a}))$ if and only if $p \cup \{\{s^M_0(\overline{a}), t^M_0(\overline{a})\}, \dots, \{s^M_{r-1}(\overline{a}), t^M_{r-1}(\overline{a})\}\}$ is a matching of $K_{n+1, n}$.
    
    For an atomic formula $R(s(\overline{x}), t(\overline{x}))$, $p \| \neg R(s(\overline{a}), t(\overline{a}))$ if and only if \\ $\{s^M(\overline{a}), t^M(\overline{a})\} \notin p$, and, for a finite conjunction $\bigwedge_{j < r} \neg R(s_j(\overline{x}), t_j(\overline{x}))$, \\ $p \| \bigwedge_{j < r} \neg R(s_j(\overline{a}), t_j(\overline{a}))$ if and only if $\{s^M_j(\overline{a}), t^M_j(\overline{a})\} \notin p$ for all $j < r$.

    Finally, for a conjunction $\bigwedge_{j < r} R(s_j(\overline{x}), t_j(\overline{x})) \land \bigwedge_{j < r} \neg R(s'_j(\overline{x}), t'_j(\overline{x}))$, $p \| \bigwedge_{j < r} R(s_j(\overline{a}), t_j(\overline{a})) \land \bigwedge_{j < r} \neg R(s'_j(\overline{a}), t'_j(\overline{a}))$ if and only if \\ $p' := p \cup \{\{s^M_0(\overline{a}), t^M_0(\overline{a})\}, \dots, \{s^M_{r-1}(\overline{a}), t^M_{r-1}(\overline{a})\}\}$ is a matching of $K_{n+1, n}$ so that $\{(s'_j(\overline{a}))^M, (t'_j(\overline{a}))^M\} \notin p'$ for all $j < r$.

    All this combined results in an $L(M)$-definable property of $\overline{a}$ equivalent to $p \| \varphi(\overline{a})$.
\end{proof}

\begin{corollary}
    There exists an $L^*$-expansion $(M, R^M)$ satisfying the principle $\neg \ontoPHP(n, R)$ and the least number principle for all existential $L^*(M)$-formulas.
\end{corollary}

\begin{remark}
    If one replaces the poset of finite matchings by a poset of definable matchings of $K_{n+1, n}$ of possibly nonstandard sizes, one can get definability of $\Vdash$ for a class of formulas with quantifiers appropriately bounded (\cite[Section 4.2]{atseriasmuller15}).

    It is also clear that there is nothing ``special" about $\neg \ontoPHP(n, R)$, and, in fact, one can apply a very similar construction and get, for example, bijective mapping between the full universe $M$ and the set of all $m < n$, or linear ordering of the set of all $m < n$ with no endpoints (see \cite{riis94} for a general criterion on which combinatorial statements are amenable for such arguments).
\end{remark}

\subsection{Framework K: Boolean-valued models with random variables}\label{sectionkrajicek}\label{sec:32}

Forcing with random variables for bounded arithmetic was developed in~\cite{krajicek11} by  Krajíček. It is conceptually similar to the Boolean-valued presentation of the independence of Continuum Hypothesis by Dana Scott~\cite{scott67}. In contrast to the approach taken in framework AM, it uses uncountable $\aleph_1$-saturated models and complete Boolean algebras to define Boolean-valued models. 

We will review this method in some detail to be able to give a self-contained presentation of the construction of the shall PHP model in Section \ref{sec:shallow}. 

We observe in Section \ref{Kint}  that, formally  speaking,  Krajíček's forcing can be interpreted in the countable framework of AM. In Section \ref{sec:KAM} we comment on the difference between the Cohen forcing and the algebra used in \cite{krajicek11}; we use the fact proved in \cite{RH:bounded} that the measure algebra used in \cite{krajicek11} has (in a fixed cardinality) an alternative description as the unique random measure algebra associated with measure space $2^\kappa$ (which we denote $\B_\kappa$).

\subsubsection{$\aleph_1$-saturation}
We  start by defining the ambient model of arithmetic. This model is a model of true arithmetic in a very rich language and its role is analogous to the role of the ``ground'' model $M$ in the previous section. We will require it to be $\aleph_1$-saturated, which is a model theoretic property stating that every consistent set of formulas in one variable containing at most countably many parameters from the model can be simultaneously satisfied by a single element of the ground model. We will not use the property directly, and we will later state two other properties $\aleph_1$-saturation implies which are used in the framework instead. To distinguish this model from the model $M$ from the previous section, we will use the calligraphic $\Mc$.

\begin{definition}
    Let $L_{all}$ be the language with names for all relations, constants and functions on the natural numbers. The ambient model $\Mc$ is a fixed $\aleph_1$-saturated model of $\{\varphi\text{ an $L_{all}$-sentence}; \N \models \varphi\}.$ We will also assume that $\N$ is equal to the well-ordered initial segment of $\Mc$, so that nonstandard numbers in $\Mc$ can be denoted $\Mc\setminus \N$. 
\end{definition}

First, let us recall an important property of non-standard models.

\begin{proposition}[Overspill]\label{propoverspill}
    The set $\N$ is not definable in $\Mc$.
\end{proposition}

Any model of true arithmetic, including $\Mc$, can be also understood as a model of finite set theory. Every number can be understood as a set of all indices of ones in its binary expansion. The subset relation $\subseteq$ is also definable, and so are set operations such as intersection, union, complement and cardinality. We say a set of elements of $\Mc$ is codeable, if it is represented by some number in $\Mc$. A form of comprehension is valid in $\Mc$, any definable subset of a codeable set is also codeable. We can similarly define codeable sequences, and a form of replacement also holds: If we have a definable function which maps a number $i$ to the $i$-th element, then any initial segment of this sequence with $t$-many elements, where $t\in \Mc$, is also a codeable sequence. We will sometimes not distinguish between numbers and the sets or sequences they represent, as we can see the set operations as being disjoint from the operations on numbers.

The $\aleph_1$-saturation implies the two properties of $\Mc$ which are important for the development of the forcing:

\begin{enumerate}
    \item Let $(a_i)_{i\in\N}$ be a sequence of elements of $\Mc$, then there is a $t\in \Mc\setminus \N$ and a sequence $(b_i)_{i<t}$ coded by an element of $\Mc$ such that $a_k=b_k$ for all $k\in \N$.
    \item Let $(A_k)_{k\in \N}$ be a sequence of definable subsets of $\Mc$ such that for all $k\in \N$ we have $\bigcap_{i<k}A_i \neq \varnothing$, then $\bigcap_{k\in \N} A_k \neq \varnothing$.
\end{enumerate}

\subsubsection{Boolean algebras}

The structures constructed from $\Mc$ will be two-sorted, Boolean-valued first-order structures. 

\begin{definition}
    A \emph{sample space} $\Omega$ is an infinite set coded by an element of~$\Mc$.
\end{definition}

From now on we will assume a specific sample space $\Omega$ is fixed. To obtain a complete Boolean algebra, we first consider the following Boolean algebra, which is provably incomplete due to the undefinability of $\N$ in non-standard models.

\begin{definition}
    Let $\A = \{A\subseteq \Omega; A\in \Mc\}$, that is $\A$ is the Boolean algebra of subsets of $\Omega$ coded in $\Mc$.
\end{definition}

A key observation in Krajíček's framework is that after taking a quotient of $\A$ by a suitable ideal, the resulting Boolean algebra is complete.

\begin{definition}\label{defibooleanalgebra}
    Let $A\in\A$, we define its \emph{probability} as the nonstandard fraction $\abs{A}/\abs{\Omega}$. We say a nonstandard fraction is \emph{infinitesimal} if its absolute value is less than any $1/k$, $k\in \N$ and we define an ideal on $\A$:
    \[\Ic=\{A \in \A; \text{ the probability of $A$ is infinitesimal}\}\]
    and we also define a Boolean algebra $\B$ as $\A/\Ic$. We also define a function 
    \begin{align*}
        \mu:\quad  \B &\to \RR \\
        [A]_{\Ic} &\mapsto \sup \: \{q \in \mathbb{Q}; q<\abs{A}/\abs{\Omega}\},
    \end{align*} where $[A]_{\Ic}$ is the equivalence class modulo $\Ic$ corresponding to an element $A\in\A$. Note that $\mu$ is well-defined precisely because $\Ic$ consists of elements of infinitesimal probability.

    We will denote the one element of $\B$ by $\1$ and the zero element by $\0$.
\end{definition}

Let us note that the value $\1$ is the only one with measure $1$ and the value $\0$ is the only one with measure $0$.

We will now present a proof of the completeness of the Boolean algebra $\B$ following the proof from~\cite[Lemma 1.2.1]{krajicek11}, with some parts of the proof rewritten in a more explicit manner.

\begin{lemma}[Key Lemma]
    The Boolean algebra $\B$ is complete.
\end{lemma}
\begin{proof}
    The proof proceeds by first establishing three claims.

    \textbf{Claim I:} The algebra $\B$ is a $\sigma$-algebra and $\mu$ is a $\sigma$-additive measure.
    \textit{Proof of claim.} Let $\{b_k\}_{k\in \N}$ be a sequence of elements of $\B$ of the form $b_k = [B_k]_{\Ic}$. Without loss of generality we can assume that $B_i\subseteq B_j$ whenever $i\leq j$ as otherwise we could consider the family \[c_k = \bigcup_{i\leq k} b_i,\]
    this implies that the sequence $\{\mu(b_k)\}_{k\in\N}$ is increasing and bounded, thus convergent and therefore Cauchy, that is for every $k\in \N $ there is $n_k\in \N$ such that 
    for ever $m\geq l \geq n_k$ we have 
    \[\frac{\abs{B_l}}{\abs{\Omega}} \leq \frac{\abs{B_{m}}}{\abs{\Omega}}\leq \frac{\abs{B_l}}{\abs{\Omega}} + \frac{1}{k},\]
    by taking the subsequence $\{b_{n_k}\}_{k\in \N}$ we can assume that $k=n_k$.

    We will find an element $C\in \A$ such that for every  $k\in \N$ we have that $B_k \subseteq C$ and $\mu([C]_{\Ic}) = \lim_{k\to \infty} \mu(b_k)$, these properties imply that $[C]_{\Ic}$ is the supremum of $\{b_k; k\in \N\}$. By the property (1) of $\Mc$ we get a sequence $\{B_i\}_{i<t}\in \Mc,$ $t\in \Mc\setminus \N$. Consider the following definable property $P(s)$:
    \[B_s\in \A \land (\forall i\leq s)\left (B_i\subseteq B_S \land \frac{\abs{B_i}}{\abs{\Omega}} \leq \frac{\abs{B_{s}}}{\abs{\Omega}}\leq \frac{\abs{B_i}}{\abs{\Omega}} + \frac{1}{i}\right),\]
    by the overspill, since it $P$ holds for every number in $\N$ there has to be some nonstandard $s_0 <t$ such that $P(s_0)$. The value $C=B_{s_0}$ then satisfies the required properties.

    \textbf{Claim II:} Any antichain in $\B$ is at most countable, that is $\B$ satisfies the~ccc condition.\\
    \textit{Proof of claim.} Let $C$ be an antichain in $\B$. By the $\sigma$-additivity of $\mu$ we can observe, that the set $C_n = \{b\in C; 1/(n+1) < \mu(b)\leq 1/n\}$ is finite for each $n\geq 1$.

    \textbf{Claim III:} For any $D\subseteq\B$ there is a countable $C\subseteq \B$ such that \[\{b \in \B; \forall c \in C : c \leq b\}= \{b\in \B; \forall d \in D: d\leq b\},\]
    that is the sets of upper bounds of $D$ and $C$ are the same.\\
    \textit{Proof of claim.} Let $D\subseteq \B$ and let $I_D$ be the ideal generated by $D$, that is the least superset of $D$ which is closed downwards and also closed under (finite) unions. It is easy to check that \[\{b\in \B; \forall d \in D: d\leq b\} = \{d\in B; \forall d\in I_D: d\leq b\},\]
    and it is well-known that a set of upper bounds of an ideal in a Boolean algebra is the same as the set of upper bounds of any maximal ideal contained within it, and by Claim II any such antichain is at most countable.

    The statement then follows by combining the claims, as any infinite supremum can be computed as a supremum over a countable set, and these suprema are available by Claim I.
\end{proof}

\subsubsection{Two-sorted Boolean-valued models}\label{sec:BooleanK}

Having defined the Boolean algebra $\B$, we are ready to define the Boolean valued structures $K(F,G)$, which are parametrized by a family of random variables $F$ and a family of random functions $G$. The universe of the structure $K(F,G)$ will consist of the ``number sort'' which is equal to $F$ and a ``set sort'' (more precisely ``function sort'') which is equal to $G$. These structures could be interpreted as generic forcing associates of $\Mc$ and their truth-values take place in $\B$.\footnote{Once a  (generic) filter on $\B$ is chosen, the truth values will collapse it to $\{0,1\}$. See Sections \ref{sec:41} and \ref{sec:42} for more details.}

In what follows all languages $L$ which we will consider are subsets of the ``maximal'' language $L_{all}$ which has (naturally interpreted) symbols for all constants, functions and relations over a fixed model of arithmetics. 

We will now define a two-sorted  extension of any language $L$ such that it will include both a number sort and a set sort. This is done so that we can in principle interpret two-sorted theories of bounded arithmetic such as $V^0_1$.

\begin{definition}
    Let $L$ be given. We define the two-sorted language $L^2$ as the language consisting of every relational and function symbol of $L$, but has two sorts of variables, the number sort (denoted by lowercase letters $x,y,z,\dots$) and the set sort (denoted by upper case letters $X,Y,Z,\dots$), where the set sort variables act like function symbols, whose arguments are from the number sort and are taking values in the number sort.
\end{definition}

Here, the ``set sort'' might be a confusing name, as in our formalism a more fitting name could be the ``function sort'', but the way the variables of the second sort are usually interpreted in theories like $V^0_1$ is as bounded sets of elements of the number sort. This can be achieved by requiring that the functions computed by the set sort variables only output $0$ or $1$.

To continue with the definition of $K(F,G)$, we will first ``truncate'' the model $\Mc$ to falsify the totality of the exponential function: without this step, it would not be possible to reinterpret the set sort because every element of the set sort might be represented by a large enough element of the number sort. To do so, we will fix a nonstandard number $n\in \Mc\setminus \N$.

\begin{definition}
    A \emph{large canonical model} is the structure
    \[\Mc_n = \{m\in \Mc;\ \exists t\in \Mc \setminus \N : m\leq 2^{n^{1/t}}\},\]
    and $L_n\subseteq L_{all}$ is the language containing all relation-symbols of $L_{all}$ and every $L$-function symbol $f$ satisfying $f[\Mc_n]\subseteq \Mc_n$.
\end{definition}

\begin{definition}
    Let $L$ be given. An non-empty set of functions $F$ is called an $L$-closed family of random variables if it satisfies:
    \begin{enumerate}
        \item $F\subseteq \Mc$, that is, every element of $F$ is a function coded by an element of $\Mc$
        \item For every $\alpha\in F$ the domain of $\alpha$ is a subset of $\Omega$ and therefore $\alpha$ is a partial function from $\Omega$ to $\Mc$. If a given $\omega\in\Omega$ is in the domain of $\alpha$, then we say $\alpha$ is defined on $\omega$, otherwise we say its undefined on $\omega$.
        \item For every $\alpha\in F$ the value \[\abs{\{\omega\in \Omega; \text{$\alpha$ is undefined on $\omega$}\}}/\abs{\Omega}\]
        is infinitesimal.
        \item For every $L$-function symbol $f$ of arity $k\in \N$ and every $\alpha_1,\dots,\alpha_k$, the partial function 
         \[f(\alpha_1,\dots,\alpha_k) : \omega \mapsto f(\alpha_1(\omega),\dots,\alpha_k(\omega))\] is in $F$ (it maps $\omega$ to $f(\alpha_1(\omega),\dots,\alpha_k(\omega))$ whenever $\alpha_1(\omega),\dots,\alpha_k(\omega)$ are defined).
    \end{enumerate}
\end{definition}

\begin{definition}
    A non-empty set of functions $G$ is called a family of random functions if every element $\Theta$ of $G$ satisfies $\Theta\in \Mc$ and to every $\omega\in \Omega$ it assigns a function $\Theta_\omega$ whose domain is a subset of $\Mc_n$ and codomain is $\Mc_n\cup\{\text{undefined}\}$. That is, $\Theta_\omega$ is a partial function from $\Mc_n$ which can include some elements of $\Mc_n$ in its domain for which it it explicitly undefined, and it is also undefined on the elements which are not in its domain.
\end{definition}

\begin{definition}
    Let a language $L$ be given. Let $F$ be an $L$-closed family of random variables and $G$ is a family of random functions. We say $G$ is $F$-compatible (or that $F$ is $G$-closed) if for every $\Theta \in G$ and every $\alpha \in F$ we have that the function
    \[\Theta(\alpha): \omega\mapsto \begin{cases}
        \text{undefined} & \text{$\alpha$ is undefined on $\omega$},\\
        \Theta_\omega(\alpha(\omega)) & \text{$\alpha$ is defined on $\omega$ and } \alpha(\omega)\in \text{dom}(\Theta_\omega),\\
        0 & \text{otherwise,}
    \end{cases}\]
    is in $F$. Note that the second case of the definition with output $\Theta_\omega(\alpha(\omega))$ can result in the output being undefined.
\end{definition}

\begin{definition}\label{def:Boolean}
    Let $L$ be a given language and let $F$ be a family of random variables and $G$ be an $F$-compatible family of random functions. We define the structure $K(F,G)$ in the language $L^2$ as the Boolean valued structure whose universe consist of two sorts, the number sort $F$ and the set sort $G$. The valuation of truth values takes place in $\B$ and is defined as follows. Note that we will define the evaluation of atomic sentences only in the case where each closed term is a member of $F$, as the $L$-closedness and $G$-closedness of $F$ makes every closed term equal to some member of $F$.
    \begin{itemize}
        \item $\bbl \alpha = \beta \bbr = \{\omega\in \Omega; \alpha(\omega)=\beta(\omega)\text{ and } \text{$\alpha(\omega)$, $\beta(\omega)$ are defined}\}/\Ic$, where $\alpha,\beta \in F$,
        \item $\bbl R(\alpha_1,\dots,\alpha_k) \bbr$ is defined as
        \[\{\omega\in \Omega; R(\alpha_1(\omega),\dots,\alpha_k(\omega)) \text{ and } \alpha_1(\omega),\dots,\alpha_k(\omega) \text{ are defined}\}/\Ic,\] where $\alpha_1,\dots,\alpha_k \in F$ and $R$ is an $k$-ary relation symbol from $L$
        \item $\bbl-\bbr$ commutes with $\land,\lor$ and $\lnot$,
        \item $\bbl(\forall x)\varphi(x)\bbr=\Land_{\alpha\in F}\bbl \varphi(\alpha)\bbr$ and $\bbl(\forall X)\varphi(X)\bbr=\Land_{\Theta\in G}\bbl \varphi(\Theta)\bbr$,
        \item $\bbl(\exists x)\varphi(x)\bbr=\Lor_{\alpha\in F}\bbl \varphi(\alpha)\bbr$ and $\bbl(\exists X)\varphi(X)\bbr=\Lor_{\Theta\in G}\bbl \varphi(\Theta)\bbr$.
    \end{itemize}

    We call any $L^2$-sentence $\varphi$ with parameters from $F$ and $G$ \emph{valid} if it satisfies $\bbl\varphi\bbr=\1$.
\end{definition}

\subsubsection{An example: The Shallow PHP Model}\label{sec:shallow}

In this section we will present as an example the construction of the shallow PHP model  $K(\FPHP,\GPHP)$ originally from~\cite[Chapter 20]{krajicek11}. Unlike in Krajíček's original book, we prove all the required results about this structure explicitly without referencing earlier parts of the book. The model is constructed so that the bijective PHP principle\footnote{The $\text{F}$ in the name of the formula $\ontoFPHP$ stands for function and is added to contrast with the $\ontoPHP$ formula which is stated for relations.}, which is the universal closure of the following formula
\begin{align*}
    \ontoFPHP(x,X)\equiv &(\exists p \leq x+1)( X(p)\geq x)\lor\\
                     &(\exists p_1, p_2\leq x+1)(p_1\neq p_2 \land X(p_1)= X(p_2))\lor\\
                     &(\exists h \leq x)(\forall p\leq x+1)(X(p)\neq h),
\end{align*}
is assigned the Boolean value $\0$, yet induction and comprehension for open $L_n^2$-formulas is preserved. The bijective PHP principle states that $X$ is not a bijective function from $\{0,\dots,x+1\}$ to $\{0,\dots,x\}$. In the context of the $\PHP$ principle, the inputs of a function are called pigeons and the outputs are called holes.

The main intuition behind the construction is as follows: To falsify the sentence $\PHP(n,\Delta)$ for some set sort object $\Delta$, we will pick $\Delta$ to be a kind of a blurred version of a randomly chosen partial injection from $\{0,\dots,x+1\}$ to $\{0,\dots,x\}$ which maps all but one pigeon injectively into all holes. The whole structure is designed so that the number sort objects in $\FPHP$ fail to uncover the missing pigeon in the domain with probability infinitesimally close to $1$, but since $\Delta$ consists of a random partial injection, they are also unable to witness a collision.

Our goal is to prove that the universal closure of $\ontoFPHP$ obtains the truth-value $\0$ in the to-be-defined structure $K(\FPHP,\GPHP)$ but every axiom from the open comprehension scheme and the open induction scheme obtains the truth-value $\1$. This section is organized as follows:

\begin{itemize}
    \item The structure $K(\FPHP,\GPHP)$ is defined in a sequence of definitions and lemmas starting with Definition~\ref{defiOPHP} and concluding with Lemma~\ref{lemmafphpgphp}.
    \item The validity of open comprehension starts with Definition~\ref{defirepresents} and concludes with Lemma~\ref{lemmaopencomp}.
    \item The validity of open induction is Lemma~\ref{lemmaopenind}.
    \item The failure of $\ontoFPHP$ consists of Definition~\ref{defidelta} and Theorem~\ref{theoremofphp}.
\end{itemize}

\begin{definition}\label{defiOPHP}
    The sample space $\OPHP$ consists of all \[\omega \subseteq \{0,\dots,n+1\}\times \{0,\dots,n\}\] such that
    \begin{enumerate}
        \item $\omega$ is a graph of a partial injection function from $\{0,\dots,n+1\}$ to $\{0,\dots,n\}$
        \item $\abs{\omega}=n$.
    \end{enumerate}

    We will use the notation $\omega(i,j)$ or $\omega(i)=j$ instead of $(i,j)\in\omega$.
\end{definition}

\begin{definition}
    A $\PHP$-tree $T$ is a finite tree whose inner nodes are labeled either by 
    \[i\to\: ?\]
    where $i \leq n+1$ or 
    \[?\to j\]
    where $j\in [n]$. If a node is labeled $i\to \: ?$ it has $n$ outgoing edges labeled by all $j\leq n$: the edge $j$ corresponds to the case when $\omega(i)=j$ holds. If a node is labeled by $?\to j$ it has $n+1$ outgoing edges labeled by all $i\in [n+1],$ with edge $i$ corresponding to the case $\omega(i)=j$.

    Any $\omega\in \OPHP$ determines a unique path in $T$ given by following the labels with edges corresponding to values of $\omega$. The path may not finish in a leaf, this happens if a query $i\to \: ?$ is posed when $i$ is not in the domain of $\omega$. If a leaf is reached we shall denote it $T(\omega)$, otherwise we leave it undefined. The depth of $T$ is equal to the maximal number of edges from root to a leaf.
\end{definition}

By induction on the depth of a $\PHP$-tree $d$, and induction on the number of restricted pigeon--hole pairs one obtains the following.

\begin{lemma}\label{lemmaphpdefined}
    Let $T$ be a $\PHP$-tree of depth $d$, then
    \[\Pr_{\omega\in\OPHP}[T(\omega) \text{ is undefined}]\leq \frac{d+1}{n+1}.\]
\end{lemma}

\begin{definition}
    A labeled $\PHP$-tree $(T,\ell)$ is a pair of a $\PHP$-tree $T$ and a labeling of leaves of $T$ by elements of $\Mc_n$. A labeled tree then computes a partial function from $\OPHP$ to $\Mc_n$, where the value at $\omega$ is the label of $T(\omega)$ whenever $T(\omega)$ is defined and undefined otherwise.
    
    We will abuse the notation somewhat and sometimes denote a labeled $\PHP$-tree simply by $T$ and its value on $\omega$ simply as $T(\omega)$ keeping the labeling implicit.
\end{definition}

\begin{definition}
    Let $\FPHP$ be the set of all partial functions \[\alpha:\OPHP\to \Mc_n\] such that there is a labeled $\PHP$-tree $T$ satisfying the following:
    \begin{itemize}
        \item For every $\omega\in\OPHP$,  $\alpha(\omega)$ is defined whenever $T(\omega)$ is defined.
        \item For every $\omega\in \OPHP$, if $T(\omega)$ is defined, then $\alpha(\omega)=T(\omega)$.
        \item The depth of $T$ is bounded above by $n^{1/t}$ for some $t\in \Mc\setminus \N$.
    \end{itemize}
\end{definition}

\begin{definition}
    Let $m\in\Mc$ and let $\hat \beta =(\beta_0,\dots,\beta_{m-1})\in \Mc$ be an $m$-tuple of elements of $\FPHP$. For every $\omega\in \OPHP$ we shall define a partial function
    \begin{align*}
        \hat \beta_\omega:\Mc_n &\to \Mc_n \cup \{\text{undefined}\}\\
        i &\mapsto \begin{cases}
        \text{undefined}&i<m\text{ and $\beta_i(\omega)$ is undefined}\\
        \beta_{i}(\omega)&i<m\text{ and $\beta_i(\omega)$ is defined}\\
        0& \text{otherwise.}
        \end{cases}
    \end{align*}

    The family $\GPHP$ consists of all random functions $\Theta$ for which there exists an $m$-tuple $\hat \beta=(\beta_0,\dots,\beta_{m-1})$ of elements of $\FPHP$, where $m\in \Mc$, such that $\Theta$ assigns to $\omega$ the function $\hat \beta_\omega$.
\end{definition}

\begin{lemma}
    The family $\FPHP$ is $L_n$-closed.
\end{lemma}
\begin{proof}
    Let $f\in L_n$ be a $k$-ary function symbol for some $k\in\N$ and let $\alpha_1,\dots,\alpha_k\in \FPHP$ computed by labeled $\PHP$-trees $T_1,\dots,T_k$ all of depth at most $n^{1/t}$ for some $t\in \Mc\setminus \N$. To compute the function $f(\alpha_1,\dots,\alpha_k)$ we can simply concatenate the labeled $T_i$ for every $1\leq i \leq k$.

    To be precise, let $T$ be the obtained by replacing every leaf of $T_1$ with a copy of $T_2$, every new leaf with a copy of $T_3$ and so on. The labels of leaves of $T$ are then taken to be simply $f(\alpha_1(\omega),\dots,\alpha_k(\omega))$ on every path which is compatible with $\omega$. The depth of $T$ is then
    \[k\cdot n^{1/t} \leq (n^{1/t})^2= n^{1/(t/2)}.\]
    Since $t/2\in \Mc\setminus\N$, we get that $T$ computes a function in $\FPHP$.
\end{proof}

\begin{lemma}\label{lemmafphpgphp}
    The family $\FPHP$ is $\GPHP$-closed.
\end{lemma}
\begin{proof}
    Consider $\alpha\in \FPHP$ computed by the labeled $\PHP$-tree $T$ and $\Theta\in\GPHP$ be computed by an $m$-tuple $\hat\beta=(\beta_0,\dots,\beta_{m-1})$ where each $\beta_i$ is computed by a labeled $\PHP$-tree $T_i$ for every $i <m$. We will show that 
    \[\Theta(\alpha): \omega\mapsto \begin{cases}
        \text{undefined} & \text{$\alpha$ is undefined on $\omega$},\\
        \Theta_\omega(\alpha(\omega)) & \text{$\alpha$ is defined on $\omega$ and } \alpha(\omega)\in \text{dom}(\Theta_\omega),\\
        0 & \text{otherwise,}
    \end{cases}\]
    is in fact in $\FPHP$.

    By overspill, there is a $t\in \Mc\setminus \N$ such that the depth of all $T,T_1,\dots,T_{m-1}$ is bounded by $n^{1/t}$. Consider the labeled $\PHP$-tree $T'$ which is obtained from $T$ by replacing every leaf labeled by $i<m$ with $T_i$ and relabeling every leaf labeled by $i\geq m$ to $0$. Consider $\omega \in \OPHP$, if $\alpha$ is undefined on $\omega$ then $T'(\omega)$ is also undefined. If $\alpha(\omega)$ is defined, then the computation of $T'(\omega)$ continues by evaluating the tree $T_{\alpha(\omega)}$. By the fact that
    \[\Theta_\omega(\alpha(\omega)) = \hat \beta_\omega(\alpha(\omega)) = \beta_{\alpha(\omega)}(\omega)\]
    we get that $T'$ computes $\Theta(\alpha)$ correctly. The depth of $T'$ is bounded by $2\cdot n^{1/t}$ which is obviously bounded above by some $n^{1/t'}$, where $t'\in \Mc\setminus \N$, and thus $\Theta(\alpha)\in \FPHP$.
\end{proof}

\begin{definition} \label{defirepresents}
    Let $f$ be a with domain $\OPHP$ and codomain $\{0,1\}$ and let $T$ be a labeled $\PHP$-tree where each label is either $0$ or $1$. We say $T$ represents $f$ if for every $\omega\in\OPHP$ where $T(\omega)$ is defined we have $T(\omega)=f(\omega).$ We say $f$ is represented by $\FPHP$ if it is represented by some tree $T$ which computes a function from $\FPHP$.
\end{definition}

\begin{lemma}\label{lemmaopenrepresent}
    Let $\varphi(y,\overline \alpha,\overline \Theta)$ be an open $L^2$-formula with parameters from $\FPHP$ and $\GPHP$ and the only free variable $y$. Then for every $i\in \Mc_n$ the function
    \begin{align*}
        f: \OPHP &\to \{0,1\}\\
           \omega & \mapsto \begin{cases}
               1 & \Mc\models \varphi(i,\overline \alpha(\omega),\overline \Theta(\omega))\\
               0 & \text{otherwise,}
           \end{cases}
    \end{align*} 
    is represented by some $\beta_i\in\FPHP$.
\end{lemma}
\begin{proof}
     By induction on the complexity of $\varphi$. For atomic $\varphi$ we can consider only formulas of the form $R(y,\alpha_1,\dots,\alpha_k)$ for some $\alpha_1,\dots,\alpha_k\in \FPHP$, because by the $\FPHP$-compatibility of $\GPHP$ we can replace any closed term built by elements of $\FPHP$ and $\GPHP$ by a single element of $\FPHP$. Let $i\in \Mc_n$, to decide whether
     \[\Mc\models R(i,\alpha_1(\omega),\dots,\alpha_k(\omega)),\tag{\dag}\]
     it is enough to glue trees computing $\alpha_1,\dots,\alpha_k$ and then labeling each leaf depending on whether (\dag) holds. Analogously to previous proofs, such a tree is still an element of $\FPHP$ because the depth increases only $k$-times for some $k\in \N$.

     If $\varphi(y)$ is of the form $\psi_1(y)\land\psi_2(y)$ it is enough to glue the trees representing the function corresponding to $\psi_1$ and $\psi_2$ and label each leaf $1$ if it was assigned the value $1$ by both of the trees. 

     Finally, if $\varphi(y)$ is of the form $\lnot\psi(y)$, we can simply swap the labels $1$ and $0$ at the leaves of the tree representing the function corresponding to $\psi$.
\end{proof}

\begin{lemma}\label{lemmaopencomp}
    For any open $L^2$-formula $A(y)$  with the only free variable $y$ with parameters from $\FPHP$ and $\GPHP$ we get that
    \[\bbl (\forall x)(\exists X)( \forall y<x)(X(y)\equiv A(y))\bbr=\1.\]
    In other words, open comprehension is valid in $K(\FPHP,\GPHP)$.
\end{lemma}
\begin{proof}
    Note that for any $\alpha \in \FPHP$ we can find $m\in \Mc_n$ such that \[\bbl\alpha<m\bbr=\1,\] by taking any $m$ which is larger than a maximal label of a leaf in the tree computing $\alpha$.

    We will proceed by showing that for every $\alpha \in \FPHP$ and $m\in \Mc_n$ satisfying $\bbl\alpha <m\bbr=\1$ we have
    \[\bbl(\exists X)( \forall y<m)(X(y)\equiv A(y))\bbr=\1,\]
    by predicate logic this implies $\bbl(\exists X)( \forall y<\alpha)(X(y)\equiv A(y))\bbr=\1$ which when proved for every $\alpha\in \FPHP$ implies
    \[\bbl (\forall x)(\exists X)( \forall y<x)(X(y)\equiv A(y))\bbr=\1.\]

    Let $\alpha\in \FPHP$ and $m\in\Mc_n$ be as stated. Take $\Theta\in \GPHP$ computed by an $m$-tuple $(\beta_0,\dots,\beta_{m-1})$ of elements of $\FPHP$ where $\beta_i$ is taken from the statement of Lemma~\ref{lemmaopenrepresent}. We get that
    \begin{align*}
        \bbl(\exists X)( \forall y<m)(X(y)\equiv A(y))\bbr &\geq \bbl (\forall y<m)(\Theta(y)\equiv A(y))\bbr\\
                                                           &= \Land_{\gamma\in\FPHP}\bbl \gamma<m \to(\Theta(\gamma)\equiv A(\gamma))\bbr\\
                                                           &= \Land_{\gamma\in\FPHP} \1 = \1,\tag{\dag}
    \end{align*}
    where $(\dag)$ follows from Lemma~\ref{lemmaopenrepresent} which implies that for every $\gamma\in \FPHP$
    \[
        \bbl \gamma<m \to(\Theta(\gamma)\equiv A(\gamma))\bbr= \1,
    \]
    this is because a truth value of a quantifier free sentence $B$ can be evaluated as the $\Ic$-equivalence class of the event
    \[E_B=\{\omega\in \OPHP; \Mc\models B\}\]
    and here $B$ taken as $\gamma<m \to(\Theta(\gamma)\equiv A(\gamma))$ is valid in $\Mc$ on every sample where all of parameters are defined which by Lemma~\ref{lemmaphpdefined} happens with probability $1-(n^{1/t}+1)/(n+1)$ which is infinitesimally close to $1$. Thus $\mu(E_B/\Ic)=1$ and $E_B/\Ic=\1$.
\end{proof}

\begin{lemma}\label{lemmaopenind}
    For any open $L^2$-formula $A(y)$  with the only free variable $y$ with parameters from $\FPHP$ and $\GPHP$ we get that
    \[\bbl (\forall x)(\lnot A(0)\lor A(x) \lor (\exists y)(A(y)\land A(y+1)))\bbr=\1.\]
    In other words, open induction is valid in $K(\FPHP,\GPHP)$.
\end{lemma}
\begin{proof}
    We will show that for every $\alpha\in \FPHP$ we can find $\beta\in \FPHP$ such that
    \[\bbl (\lnot A(0)\lor A(\alpha) \lor (A(\beta)\land A(\beta+1)))\bbr=\1,\]
    it is not hard to see that this implies the validity of open induction.

    Let $\alpha\in \FPHP$ and let $\alpha_1,\dots,\alpha_k\in \FPHP$ be the parameters of $A$, we again omit parameters from $\GPHP$ as the closed terms they form with elements of $\FPHP$ can be evaluated as elements of $\FPHP$.

    We will construct the tree computing $\beta$ to behave as follows: It first decides whether $\lnot A(0)$ or $A(\alpha)$ are true, if they are, then the tree terminates with an arbitrary output. Otherwise, it performs binary search by keeping track of pairs $(i_t,j_t)$, with $(i_0,j_0)=(0,a)$ where $a$ is the value of $\alpha$ on the given sample and each $(i_{t+1},j_{t+1})$ is either $(i_{t},\lfloor (i_t+j_t)/2\rfloor )$ or $(\lfloor (i_t+j_t)/2\rfloor, j_t)$ depending on which of the pairs satisfy that $A(i_{t+1})$ is true and $A(j_{t+1})$ is not. In $\log(a)$ many steps it finds a pair $(i,i+1)$ such that $A(i)$ is true and $A(i+1)$, making the induction axiom for $A$ true on the given sample.

    To construct the tree $T'$ computing $\beta$ we will use the trees $\beta_i$ from the statement of Lemma~\ref{lemmaopenrepresent}. To decide the validity of $\lnot A(0)$ it is enough to use $\beta_0$. Then, to decide the validity of $A(\alpha)$, we prolong $T'$ along each path with the tree computing $\alpha$ and append to every leaf labeled by $i$ the tree $\beta_i$. Now, for every path where $\lnot A(0)$ is true or $A(\alpha)$ is true, we are done. The rest of the paths are then extended to compute the binary search procedure described above. To do so, we extend each path with the tree $\beta_{(i_0+j_0)/2}$ and depending on its answer we define $(i_1,j_1)$ for the given path and continue with the tree $\beta_{(i_1+j_1)/2}$ and so on until we get to the case where $j_t=i_t+1$, in this case we label the leaf by $i_t$.

    By construction, the tree computes the correct function, it remains to consider the depth of the tree. We know that there is $s\in \Mc\setminus \N$ such that the trees computing $\alpha,\alpha_1,\dotsm,\alpha_k,\beta_0,\dots,\beta_{m-1}$ are all of depth bounded above by $n^{1/s}$, here $m\in \Mc_n$ is picked to be strictly larger than $\alpha(\omega)$ for any $\omega\in \OPHP$. The value $t\in \Mc$ where the binary search terminates is bounded by $\log(a)$ which is bounded by $\log(m)$. The full depth of the tree is then bounded by 
    \[(3+t)\cdot n^{1/s} \leq (3+\log(m))\cdot n^{1/s}\leq n^{1/s'}\cdot n^{1/s}\leq n^{1/s''},\]
    for some $s',s''\in \Mc \setminus \N$, where the inequality $\log(m)\leq n^{1/s'}$ for some $s'$ follows from $m$ being a member of the large canonical model $\Mc_n$,
\end{proof}

\begin{definition}\label{defidelta}
    We define an element $\Delta\in\GPHP$ to be computed by the $(n+1)$-tuple $(\delta_0,\dots,\delta_n)$, where $\delta_i$, $i<n+1$, is computed by the depth $1$ labeled $\PHP$-tree with the root labeled as $i\to\: ?$ and each leaf labeled the same as the edge incident to it. In other words, $\delta_i(\omega)$ is the value $j$ such that $\omega(i)=j$ whenever such a value exists, otherwise it is undefined. 
\end{definition}

\begin{theorem}\label{theoremofphp}
    In $K(\FPHP,\GPHP)$ it holds that 
    \[\bbl\ontoFPHP(n,\Delta)\bbr = \0.\]
    In other words, it is valid in $K(\FPHP,\GPHP)$ that $\Delta$ is a bijective function from $\{0,\dots,n+1\}$ to $\{0,\dots,n\}$.
\end{theorem}
\begin{proof}
    We will prove that the negation of each of the four disjuncts is assigned the value $\1$. We will start with the second disjunct which when negated becomes
    \[(\forall p_1, p_2\leq x+1)(p_1\neq p_2 \to (X(p_1) \neq  X(p_2))\]
    we will show for every $\alpha_1,\alpha_2$ that 
    \[\bbl(\alpha_1<n+1\land \alpha_2 <n+1 \land \alpha_1\neq \alpha_2) \to (\Delta(\alpha_1) \neq \Delta(\alpha_2)) \bbr=\1,\]
    since this is a quantifier free sentence, it is enough to check that the event consisting of all $\omega$ where all of the random variables in the formula are defined and the formula is true in $\Mc$ has probability infinitesimally close to $1$. Since every $\omega \in \OPHP$ is a partial injective function, it maps each pigeon to a different hole, and thus the formula is true on every sample where all of the random variables are defined. By the union bound on the probabilities of $\alpha_1$,$\alpha_2$,$\Delta(\alpha_1)$ and $\Delta(\alpha_2)$ being defined and Lemma~\ref{lemmaphpdefined} we get that this event has probability infinitesimally close to $1$.

    Regarding the first disjunct, whose negation is
    \[(\forall p \leq n+1)(\Delta(p)<n),\]
    we will show that for every $\alpha\in\FPHP$
    \[\bbl\alpha<n+1 \to\Delta(\alpha)<n\bbr=\1,\]
    the sentence we are evaluating is again quantifier free and on every sample $\omega\in\OPHP$ where both $\alpha$ and $\Delta(\alpha)$ are defined, it is true, because $\omega$ is a partial injective function from $\{0,\dots,n+1\}$ to $\{0,\dots,n\}$. Again by union bound and Lemma~\ref{lemmaopenrepresent} we get that the probability of both $\alpha$ and $\Delta(\alpha)$ being defined is infinitesimally close to $1$.

    Regarding the third disjunct, whose negation is 
    \[(\forall h \leq n)(\exists  p\leq n+1)(\Delta(p) = h)\]
    we will show that for every $\beta \in \FPHP$ there is $\alpha\in \FPHP$ such that 
    \[\bbl \beta\leq n\land \alpha\leq n+1\land  \Delta(\alpha) = \beta\bbr=\1,\]
    here $\alpha$ is computed a labeled $\PHP$-tree which starts off as the tree computing $\beta$ whose every leaf with label $j\leq n$ we extend by a query $?\to j$ and label each new leaf with the label of the edge incident to it. By the fact that each $\omega\in \OPHP$ is surjective the quantifier free sentence is true on every sample where $\beta$, $\alpha$ and $\Delta(\alpha)$ are defined, which is by an analogous argument an event of probability infinitesimally close to $1$.
\end{proof}

\subsection{Framework U: Countable restricted ultrapowers}\label{restricted}

Suppose $M$ is a model of arithmetic, $I \sub M$ a set of indexes, $F$ a collection of some functions $f: I \to M$, and $\mathcal{U}$ an ultrafilter on $I$. An ultrapower of $M$ by $\mathcal{U}$ which uses only functions $f \in F$ is called  \emph{an $F$-restricted ultrapower of $M$} and denoted $F/\mathcal{U}$. In the setting of bounded arithmetic, $F$ will always be a subset of functions \emph{definable} over the structure $M$. Moreover, in this section, $M$ and $F$  will always be countable and hence the resulting ultrapowers will  be countable as well. 

While seemingly unrelated to forcing constructions, we will  show in Section \ref{sec:ultrapowerforcing} that (countable) restricted ultrapowers can interpreted as a forcing construction in the sense of framework AM. See Section \ref{sec:42} where we use set theory to interpret  generalizations to uncountable ultrapowers.

In the first Subsection \ref{Skolem} we discuss the situation where $F$ is the set of all definable functions.  A similar construction was used already by Skolem~\cite{skolem1933uber} to give the first explicit countable non-standard model of arithmetic.\footnote{Skolem's model also uses equivalence classes of definable sequences of elements, but the equivalence relation was defined using the concept of cohesive sets. This was generalized to the concept of Cohesive Powers by Dimitrov~\cite{dimitrov2009cohesive}.} While this is not immediately relevant to forcing in bounded arithmetic, it is considered the historically first construction of a non-standard model of arithmetic and nicely ties to the more general setting of restricted ultrapowers.

In the second Subsection \ref{F}, we give an example of a more general $F$-restricted ultrapower where $F$ is the family of polynomial time computable functions. This construction was used by Bydžovský and Müller~\cite{bydvzovsky2020polynomial} to obtain a model of a theory $T_{PV}$.\footnote{They show that a certain computational complexity holds which is not known to hold in the standard model. Their result extends a previous result of Krajíček and Oliveira~\cite{krajicek2017unprovability} which was proven without direct model theoretic constructions.} Another application of restricted ultrapowers to bounded arithmetic was given by Garlík~\cite{garlik2016construction}.

\subsubsection{Ultrapowers with definable functions} \label{Skolem}

\begin{definition}
    Let $L$ be a language, $M$ a countable $L$-structure and let $I\subseteq M$ be countable. Define $L(I)$ to be an extension of $L$ with a unary predicate interpreted as $I$ in $M$. Let $M^{(I)}$ be the set of all functions whose domain is $I$ and whose graph is definable by an $L(I)$-formula with parameters from $M$.
\end{definition}

\begin{definition}
    Let $L$ be a language, $M$ a countable $L$-structure and let $I\subseteq M$ be countable.

    Let $\tilde{\mathcal{P}}(I)$ be the Boolean algebra of all $L(I)$-definable subsets of $I$ with parameters from $M$. Let $\mathcal{U}$ be an ultrafilter on $\tilde{\mathcal{P}}(I)$. The ultrapower $M^{(I)}/\mathcal{U}$ is an $L$-structure defined in the following way:
    \begin{itemize}
        \item the universe is $M^{(I)}/\sim$, where \[\alpha \sim \beta\text{ iff }\{i\in I;M \models \alpha(i)=\beta(i)\} \in \mathcal{U},\]
        \item for every $k$-ary function symbol $f\in L$, $f(x_1,\dots,x_k)$ is interpreted as $([\alpha_1]_{\sim},\dots,[\alpha_k]_{\sim}) \mapsto [f(\alpha_1,\dots,\alpha_k)]_\sim$,
        \item for every $k$-ary relation symbol $R\in L$, $M^{(I)}/\mathcal{U}\models R(\alpha_1,\dots,\alpha_k)$ iff $\{i\in I; M\models R(\alpha_1(i),\dots,\alpha_k(i))\}\in \mathcal{U}$.
    \end{itemize}

    We say $M^{(I)}$ is \emph{full} if for every $L$-formula $\varphi$, every $\beta_1,\dots,\beta_k\in M^{(I)}$ we have $\alpha \in M^{(I)}$ such that for every $i\in I$ 
    \begin{equation}\label{O:full} M\models (\exists x)\varphi(x,\beta_1(i),\dots,\beta_k(i))\to\varphi(\alpha(i),\beta_1(i),\dots\beta_k(i)).
   \end{equation}

    For any $L$-formula $\varphi$ and $\alpha_1,\dots,\alpha_k\in M^{(I)}$ we denote 
    \[\bbl\varphi(\alpha_1,\dots,\alpha_k)\bbr = \{i \in I; M\models \varphi(\alpha_1,\dots,\alpha_k)\}.\]
\end{definition}

The following is the fundamental theorem of ultrapowers and implies the corollary that whenever $M^{(I)}$ is full, ultrapowers composed of definable functions are elementarily equivalent to full ultrapowers. It also implies that whenever $M^{(I)}$ is full, $\bbl-\bbr$ is a truth-valuation of a Boolean-valued model. See Section~\ref{sec:42} for more details on full Boolean models in an uncountable setting.

\begin{theorem}[Łoś]\label{O:Los}
    Let $L$ be a language, $M$ a countable $L$-structure and $I\subseteq M$ such that $M^{(I)}$ is full. Let $\mathcal{U}$ be an ultrafilter on the Boolean algebra $\tilde{\mathcal{P}}(I)$ defined as above. Then for every $L$-formula $\varphi(x_1,\dots,x_k)$ and all $\alpha_1,\dots,\alpha_k \in M^{(I)}$:
    \begin{align*}
        M^{(I)}/\mathcal{U}\models \varphi([\alpha_1]_{\sim},\dots,[\alpha_k]_{\sim}) \iff \bbl \varphi(\alpha_1(i),\dots,\alpha_k(i))\bbr \in \mathcal{U}
    \end{align*}
\end{theorem}
\begin{proof}
The proof is routine, using the fullness to deal with the existential step.
\end{proof}

\begin{corollary} \label{O:Losc}
    Let $L$ be a language, $M$ a countable $L$-structure and $I\subseteq M$, such that $M^{(I)}$ is full. Let $\mathcal{U}$ be an ultrafilter on the powerset $\mathcal{P}(I)$ and $\mathcal{U}'$ its restriction to $\tilde P(I)$, let $\sim$ be the equivalence relation on $M^I$ defined by $\mathcal{U}$ and $\sim'$ the equivalence relation $M^{(I)}$ defined by $\mathcal{U}'$. Then there is an elementary embedding $e:M^{(I)}/\mathcal{U}'\to M^I/\mathcal{U}.$
\end{corollary}
\begin{proof}
    %\textbf{Claim:} For every $\alpha\in M^{(I)}$, $[\alpha]_{\sim'} = \{\beta\in[\alpha]_{\sim};\beta \in M^{(I)}\}$.

    %\emph{Proof of claim.} For any $\alpha,\beta\in M^{(I)}$ we have that 
    %\[\{i\in I; \alpha(i)=\beta(i)\}\in \mathcal{U}' \Leftrightarrow \{i\in I; \alpha(i)=\beta(i)\}\in \mathcal{U},\]
    %simply by the fact that $\{i\in I; \alpha(i)=\beta(i)\}\in \tilde{\mathcal{P}}(I)$. Then
    %\[\alpha \sim' \beta \Leftrightarrow \alpha\sim \beta,\]
    %which implies the claim.

    Clearly, $e:[\alpha_i]_{\sim'}\mapsto [\alpha_i]_{\sim}$ is an embedding. By the Łoś theorem for definable ultrapowers and ultrapowers:
    \begin{align*}
        M^{(I)}/\mathcal{U}'\models \varphi([\beta_1]_{\sim'},\dots,[\beta_k]_{\sim'}) &\Leftrightarrow \bbl\varphi(\beta_1(i),\dots,\beta_k(i))\bbr \in \mathcal{U}'\\
        &\Leftrightarrow\bbl\varphi(\beta_1(i),\dots,\beta_k(i))\bbr \in \mathcal{U}\tag{*}\label{tagdefinable}\\
        &\Leftrightarrow M^{I} \models \varphi([\beta_1]_\sim,\dots,[\beta_k]_\sim),
    \end{align*}
    where (\ref{tagdefinable}) follows from $\bbl\varphi(\beta_1(i),\dots,\beta_k(i))\bbr \in {\tilde{\mathcal{P}}(I)}$.
\end{proof}

In the following theorem we mention a result of Skolem formulated in  the modern framework of ultrapowers. We refer the interested reader to~\cite{Dimitrov2020} for a description of Skolem's original model.

\begin{theorem}[Skolem]
    Consider $\mathbb{N}$ in the language $L_{PA}=\{0,1,+,\cdot ,\leq\}$, let $L=L_{PA}$ and $M=I=\mathbb{N}$. Then $\mathbb{N}^{(\mathbb{N})}$ is full, and for every non-principal ultrafilter $\mathcal{U}$ on $\tilde{\mathcal{P}}(I)$ the structure $\mathbb{N}^{(\mathbb{N})}/\mathcal{U}$ is a countable elementary extension of $\mathbb{N}$ which is not isomorphic to $\mathbb{N}$.
\end{theorem}
\begin{proof}
The fullness follows easily from the definability of the least witness for valid existential formulas. The remaining part of the proof is standard.
\end{proof}

\subsubsection{Restricted Ultrapowers}\label{F}

In the context of this article, restricted ultrapowers are defined by means of families $F$ of functions which are all definable in the underlying model. This context is sufficiently general for the purposes of this article (the number sort in \Krno's forcing with random variables can be analogously, using uncountable structures, see Definition \ref{def:ultrap}).

\begin{definition}\label{def:restricted-ultrapower}
    Let $L$ be a language, $M$ a countable $L$-structure, a countable $I\subseteq M$ and let $F\subseteq M ^{(I)}$ such that it is \emph{$L$-closed}, that is, for every $k$-ary function symbol $f\in L$, and every $\alpha_1,\dots,\alpha_k \in F$, we have $(i\mapsto f(\alpha_1(i),\dots,\alpha_k(i)))\in F$. Let $\mathcal{U}$ be an ultrafilter on $\tilde{\mathcal{P}}(I)$, then we define the restricted ultrapower $F/\mathcal{U}$ as an $L$-structure in the following way: 
    \begin{itemize}
        \item the universe is $F/\sim$, where $\alpha \sim \beta$ if $\{i\in I;\alpha(i)=\beta(i)\} \in \mathcal{U}$,
        \item for every $k$-ary function symbol $f\in L$, $f(x_1,\dots,x_k)$ is interpreted as $([\alpha_1]_{\sim},\dots,[\alpha_k]_{\sim}) \mapsto [f(\alpha_1,\dots,\alpha_k)]_\sim$,
        \item for every $k$-ary relation symbol $R\in L$, $$F/\mathcal{U}\models R(\alpha_1,\dots,\alpha_k) \iff \{i\in I; R(\alpha_1(i),\dots,\alpha_k(i))\}\in \mathcal{U}.$$
    \end{itemize}
 We call $F/\mathcal{U}$ an \emph{$F$-restricted ultrapower}.
\end{definition}

One can check easily that if $F/\mathcal{U}$ is full in the sense of (\ref{O:full}), then Łoś theorem \ref{O:Los} and Corollary \ref{O:Losc} still apply.

Bydžovský and Müller used restricted ultrapowers to obtain models of the theory $T_{PV}$. We will not describe the theory $T_{PV}$ in detail, but it can be understood as the theory of all true universal arithmetical sentences in a language which contain a function symbol for every Turing machine whose computations terminate after polynomially many steps in the length of the input, and such a symbol represent the function computed by this Turing machine. It extends a subtheory of $S^1_2$ denoted $PV_1$.

\begin{theorem}[{\cite[Corollary 1.5]{bydvzovsky2020polynomial}}]\label{pvcomplete}
    Every countable model of $T_{PV}$ is isomorphic to an $F$-restricted ultrapower $F/\mathcal{U}$ for a family $F$ of polynomial time computable functions $\N \to \N$ and an ultrafilter $\mathcal{U}$ on $\mathbb{N}$.
\end{theorem}

Therefore, while it is not known how to use restricted ultrapowers to obtain unprovability of universal formulas (as all true universal formulas are valid in models of $T_{PV}$), we know that the method is strong enough to yield unprovability over $T_{PV}$ of any formula of higher quantifier complexity.

We will now state the result of Bydžovský and Müller obtained by restricted ultrapowers about a complexity theoretic statement true in an ultrapower but unknown to hold in the standard model. The statement in question is of the following form: `Every set decidable in polynomial time can by decided by a sequence of Boolean circuits with at most $O(n^k)$ gates, for some $k$, where $n$ is the length of the input.' For precise definition of these notions, we refer the interested reader to~\cite{arora2009computational}. The statement of the theorem mentions the notion of a \emph{Herbrand saturated} model (see~\cite{avigad2002saturated}). The precise meaning is not important for our exposition; we note that every model of a universal theory can always be $\exists$-elementarily extended to an Herbrand saturated model. Let $g$ be a function symbol in the language of $T_{PV}$, and let $LB[g](s, n)$ be a formula in the language of $T_{PV}$ with variables $s$ and $n$ that expresses the property ``For all circuits $C$ of size at most $s$ there is $x$ of bit-length $n > 0$ such that $g(x) \neq C(x)$.''

\begin{theorem}[{\cite[Theorem 1.1]{bydvzovsky2020polynomial}}]
    Let $k\in \N, 0<\epsilon<1$ be rational, and $M$ a Herbrand saturated model of $T_{PV}$. Then there is restricted ultrapower $F/\mathcal{U}$, $F\subseteq M^{(M)}$, consists only of polynomial time computable functions, such that
$F/\mathcal{U}\models \forall z \exists \ell LB[g](\abs{z}\cdot \ell ^k,\ell).$
\end{theorem}

The previous result can be interpreted as the unprovability of the inclusion of complexity classes $\text{P}\subseteq \text{SIZE}[n^k]$ for any fixed $k$ in the theory $T_{PV}$, where $\text{P}$ is the complexity class of polynomial time decidable problems and $\text{SIZE}[n^k]$ is the class of all problems decidable with a family of circuits of size $O(n^k)$. It seems that there is a wide agreement among theorists that indeed $\text{P}\not \subseteq \text{SIZE}[n^k]$, but this is not known to hold in the standard model.

\section{Interpreting frameworks K and U in framework AM}\label{sectionrelationship}\label{sec:33}

\subsection{Complete Boolean algebras and the countable L{\"o}wenheim--Skolem argument}\label{sec:KAM}

As we will observe, both frameworks K and U can be interpreted in AM by first formulating both as Boolean-valued models with respect to certain complete (uncountable) Boolean algebras $\Bk$ and $\CoRO$\footnote{The Boolean completion (the regular open algebra) of the Cohen forcing.}, respectively, and then use the standard set-theoretic method of defining a forcing relation between elements of the Boolean algebra and sentences in the forcing language by postulating $$b \Vdash \varphi \iff b \le \bv{\varphi}.$$ There are two additional tasks to deal with:

\begin{itemize}
\item
Use a L{\"o}wenheim--Skolem argument to collapse the model $\M$ and the algebra $\Bk$ to a countable subalgebra $\Bk'$ for framework K, and collapse $\CoRO$ to a countable subalgebra $\CoRO'$ for framework U. This ensures that the resulting forcing satisfies the countability condition postulated by AM.
\item
To extend the framework AM to accommodate the two-sorted language in framework K.
\end{itemize}

Importantly, while the framework AM can formally interpret Krajíček's forcing and restricted ultrapowers, it does not help with verifying the validity of formulas: the principal theorem \ref{principal theorem} does not apply because \Krno's forcing and restricted ultrapowers are not conservative. Yet, one can interpret this result as an indication that the definition of forcing given by \AMno~is sufficiently general for the current applications; moreover it shows the the forcing framework  is complete with regards to constructing models of $T_{PV}$  because it interprets the ultrapower construction which we know to be sufficiently general (Theorem~\ref{pvcomplete}). 

\textbf{The random algebra $\Bk$.} 
Despite its non-standard definition, the complete Boolean algebra $\B$ from Section \ref{sec:32} is isomorphic to a probability measure algebra $\Bk$ (also a ``random algebra'') defined on the space $2^\kappa$ for some uncountable $\kappa$.\footnote{Let $\kappa$ be an infinite cardinal and let $T$ be the set of all finite functions from $\kappa$ to $2$. Let $\Sigma_\kappa$ be a $\sigma$-algebra generated by the system $\set{S_t}{t \in T}$, where $S_t = \set{f \in 2^\kappa}{t \sub f}$. The product measure $\lambda_\kappa$ is the unique $\sigma$-additive measure on $\Sigma_\kappa$ such that $\lambda_\kappa(S_t) = \frac{1}{2^{|t|}}$. The random algebra $\Bk$ is the quotient $\Sigma_\kappa/\mbox{Null}$, where Null is the ideal of null sets with respect to measure $\lambda_\kappa$. See \cite{RH:bounded} for more details.}  This follows from Maharam theorem \cite{maharam1942homogeneous} and theorems in non-standard analysis in \cite[Lemma 2.2]{jin1992maharam} (see \cite[Section 5.2]{RH:bounded} for more details). The cardinal $\kappa$ in $2^\kappa$ depends on the size of the starting model $\M$ of true arithmetic.\footnote{The exact size of $\M$ is not relevant for the purposes of the calculations in Section \ref{sec:32}, the only requirement is that $\M$ is $\aleph_1$-saturated.} To have a concrete representation, we can start with $\CH$ (the Continuum Hypothesis) in which case $\M$ can be chosen to be  an $\aleph_1$-saturated model of size $\aleph_1$ (and hence fully saturated). By a general theorem,  if $M,N$ are elementarily equivalent first-order structures of size $\kappa$ which are fully saturated, then they are isomorphic (see \cite[Theorem 10.1.8]{MT:Hod}). It follows that $\M$ is up to isomorphism the unique model of true arithmetic which is $\aleph_1$-saturated and of size $\aleph_1$. In this case, $\B$ is isomorphic to the algebra generated by the null ideal on the product measure space $2^{\aleph_1}$ (we denote it $\Bi$). This algebra is well-known and used in set theory to add $\aleph_1$-many new reals (real numbers, or equivalently, subsets of natural numbers). See Section \ref{sec:Maharam} where we compare the use of this algebra in set theory and in arithmetic. 

\textbf{The Cohen algebra $\CoRO$.} It is well-known that there is up to isomorphism just a single countable atomless Boolean algebra. Let $\CoRO$ denote its completion. We call $\CoRO$ the Cohen algebra. More generally, any countable dense suborder of $(\CoRO,<)$, where $<$ is the canonical ordering on $\CoRO$, is called the Cohen forcing. A typical representation of Cohen forcing is the set of all finite sequences of zero and ones, ordered by reverse inclusion. For the purposing of interpreting forcing with countable restricted ultrapowers, we will consider a less common representation which we denote $\Pfin$ for some countable subset $I$ of a countable model $M$ of arithmetic.  $\Pfin$ is an atomless countable Boolean algebras containing equivalence classes of countably many subsets of $I$,  quotioned by the ideal of (standardly) finite sets.

\brm
We find it interesting that the known forcings in arithmetic are derived just from two types of Boolean algebras: $\CoRO$ and $\Bk$. Moreover, after the identification with restricted ultrapowers via generic ultrafilters on $\CoRO$ and $\Bk$, the variability of the forcing arguments seems more restricted still. This contrasts with the much wider range of Boolean algebras used in set theory (see the discussion concerning  the internal and external perspectives in set theory regarding the variability of forcing notions in Section \ref{sec:1}).
\erm

\subsection{Framework K: Boolean-valued models and the random algebra $\B_\kappa$}\label{Kint}
The main goal of this section is to show that if we overcome certain cardinality-related obstacles we can show that the structures $K(F,G)$ define a forcing in the sense of Definition~\ref{forcing}.

The usual way of interpreting a Boolean-valued structure as a forcing notion (in the sense of Section \ref{sec:partially definable forcing}) is to take $\B\setminus\{\0\}$ as the forcing frame and define $b\Vdash\varphi$ iff $b\leq \bbl\varphi\bbr$. However, it is well known that a complete non-atomic Boolean algebra cannot be countable (it has size at least $2^{\aleph_0}$) which is a requirement for a forcing frame of the AM framework. We start by shrinking the $K(F,G)$ construction to a countable setting using the Löwenheim--Skolem theorem.

\begin{lemma}\label{lemmakfgshrink}
    Let  $n\in \Mc\setminus \N$, $L\subseteq L_{n}$ countable, let $\Omega$ be a sample space, let $F$ be an $L$-closed family of random variables, let $G$ be an $F$-compatible family of random functions and let $\B$ be the Boolean algebra as in Definition~\ref{defibooleanalgebra}.
    
    Then there is a countable $F'\subseteq F$, a countable $G'\subseteq G$, a countable Boolean subalgebra $\B'\subseteq \B$ and a valuation $\bbl-\bbr'$ of $L^2$-formulas with parameters from $F'$ and $G'$, such that they satisfy
    \begin{itemize}
        \item for every closed $L^2$-term $t$ with parameters from $F'$ and $G'$ there is a $\alpha\in F'$ such that
        $\bbl t=\alpha\bbr' = \1$,
        \item $\bbl-\bbr'$ commutes with $\lnot$, $\lor$, $\land$,
        \item for every $L^2$-formulas $\varphi(x)$ and $ \psi(X)$ with parameters from $F'$ and $G'$ each of the sets $\{\bbl\varphi(\alpha)\bbr'; \alpha\in F'\}$ and $\{\bbl\varphi(\Theta)\bbr'; \Theta \in G'\}$ have both a supremum and an infimum in $\B'$,
        \item $\bbl(\forall x)\varphi(x)\bbr' = \Land_{\alpha\in F'} \bbl\varphi(\alpha)\bbr'$ and $\bbl(\forall X)\psi(X)\bbr' = \Land_{\Theta\in G'} \bbl\psi(\Theta)\bbr'$, where $\varphi(x)$ and $\psi(X)$ are $L^2$-formulas with parameters from $F'$ and $G'$,
        \item $\bbl(\exists x)\varphi(x)\bbr' = \Lor_{\alpha\in F'} \bbl\varphi(\alpha)\bbr'$ and $\bbl(\exists X)\psi(X)\bbr' = \Lor_{\Theta\in G'} \bbl\psi(\Theta)\bbr'$, where $\varphi(x)$ and $\psi(X)$ are $L^2$-formulas with parameters from $F'$ and $G'$,
        \item $\bbl\varphi \bbr'=\bbl\varphi\bbr$, where $\varphi$ is an $L^2$-sentence with parameters from $F'$ and $G'$.
    \end{itemize}
    We shall denote the Boolean valued $L^2$-structure with the universe consisting of the number sort $F'$ and the set sort $G'$ and the $\B'$-valuation $\bbl-\bbr'$ by $K'(F',G')$.
\end{lemma}
\begin{proof}
    Consider the structure $S$ whose universe consists of the number sort $F$, the set sort $G$ and the Boolean algebra sort $\B$, and whose language is
    \[L_S =\{\text{val}_\varphi; \varphi\text{ is an $L^2$-formula}\} \cup \{\leq \},\]
    where for each $L^2$-formula $\varphi(\overline x,\overline X)$ we define the interpretation of $\text{val}_\varphi$ to be the function whose arguments are from the number sort and the set sort and whose output is from the Boolean algebra sort
    \[\text{val}_{\varphi}:(\overline \alpha, \overline \Theta) \mapsto \bbl \varphi(\overline \alpha, \overline \Theta)\bbr,\]
    and $\leq$ is the partial order of the Boolean algebra.
    
    Note that $L_S$ is countable and therefore we can use the Löwenheim--Skolem theorem to obtain a countable elementary substructure $S'$ of $S$, whose universe consists of the number sort $F'$, the set sort $G'$ and the Boolean algebra sort $\B'$.
    
    We now define a $\B'$-valued structure $K'$ whose universe consists of the number sort $F'$ and the set sort $G'$ and whose valuation of $L^2$-formulas is defined by    \[\bbl \varphi (\overline \alpha,\overline \Theta)\bbr' = \text{val}_{\varphi}(\overline \alpha,\overline \Theta).\]
    By the elementarity of $S'$ the resulting structure is as required.
\end{proof}

Let us note that this demonstrates that the $\aleph_1$-saturation of $\Mc$, which is used for the completeness of the Boolean algebra $\B$, is not essential for the (un)provability consequences of the structures $K(F,G)$. A countable model is sufficient;\footnote{More formally, this corresponds to constructing a countable elementary submodel $(N,\in) \preceq (H(\kappa),\in)$ of $(H(\kappa),\in)$ for some sufficiently large regular uncountable cardinal $\kappa$, with $N$ containing an uncountable $\aleph_1$-saturated model $M$ and the complete Boolean algebra $\B$ (and other relevant objects). Note that $H(\kappa)$ denotes the set of all sets $x$ with transitive closure of size $<\kappa$; it is known that $(H(\kappa), \in)$ is a transitive model of $\mathrm{ZFC}$ minus the powerset axiom, and thus supports all the usual constructions with fixed objects performed in the ambient set theory.  Let $\pi: N \to \bar{N}$ be the transitive collapse of $(N,\in)$ onto a countable transitive model $(\bar{N},\in)$ of set theory. Then $\pi(M)$ and $\pi(\B)$ in $\bar{N}$ are countable and rich enough to support the required Boolean-valued calculations performed in $H(\kappa)$ (and hence in the ambient set theory). Even more generally, one can start with a countable transitive model $(V,\in)$ of set theory, and choose in $V$ an $\aleph_1$-saturated (in $V$) model of arithmetic $M$ and work entirely inside $V$, bypassing the the L{\"o}wenheim--Skolem construction and the transitive collapse. For the purposes of showing independence of sentences in bounded arithmetic, this is entirely sufficient (see \cite[Remark 2.4]{RH:bounded}).} however, we should like to add that this reduction is purely a formalistic exercise in the L{\"o}wenheim--Skolem theorems, the real intuition behind the construction is founded on saturated models. 

The previous Lemma is enough to obtain a countable poset. The only remaining difficulty is that Krajíček's forcing is defined for two-sorted languages $L^2$ and these are not handled explicitly by Atserias and Müller. In the following we extend the Definition~\ref{forcing} to the two-sorted setting.

\begin{definition}[Universal Forcing for two-sorted structures.]
\label{twosortedforcing}
    Let $L$ be a countable language, $L*$ its countable extension and a non-empty countable partially ordered set $(P,\preceq)$. Let $S=(S_0,S_1)$ be a countable $L^2$-structure and $L^*(M)$ be the extension of $L^*$ by all constants for every member of $M_0$ and $M_1$. A binary relation $\Vdash$ between elements of $P$ and $(L^*(M))^2$-sentences is called a \deff{forcing} if it satisfies the following properties for all sentences $\varphi, \psi$, all formulas $\chi(x)$, and all conditions $p, q$.
    \begin{enumerate}
        \item \deff{Universal pre-forcing recurrence}:
        \begin{itemize}
            \item $p \Vdash \neg \varphi \:$ if and only if  $\: \forall q \preceq p: \: q \nVdash \varphi$;
            \item $p \Vdash \varphi \land \psi \:$ if and only if  $p \Vdash \varphi$ and $p \Vdash \psi$;
            \item $p \Vdash \varphi \lor \psi \:$ if and only if $\: p \Vdash \neg(\neg \varphi \land  \neg \psi)$;
            \item $p \Vdash \forall x \chi(x) \:$ if and only if, for all $a \in S_0$, $\: p \Vdash \chi(a)$;
            \item $p \Vdash \exists x \chi(x) \:$ if and only if $\: p \Vdash \neg \forall x \neg \chi(x)$;
            \item $p \Vdash \forall X \chi(X) \:$ if and only if, for all $A \in S_1$, $\: p \Vdash \chi(A)$;
            \item $p \Vdash \exists X \chi(X) \:$ if and only if $\: p \Vdash \neg \forall X \neg \chi(X)$.
        \end{itemize}

        \item \deff{Extension}: if $p \Vdash \varphi$ and $q \preceq p$, then $q \Vdash \varphi$; i.e. the set $[\varphi]$ is \deff{downwards closed}.

        \item \deff{Stability}: if $\forall q \preceq p \: \exists r \preceq q \: : \: r \Vdash \varphi$, then $p \Vdash \varphi$.
    \end{enumerate}
\end{definition}

Here is the theorem with the extension to two-sorted langauges.\footnote{The theorem was independently proved by Johannes Heil during Radek Honzik's visit in Passau in May 2026.}

\begin{theorem}\label{theoremkrajicekismuller}
    Assume we have a structure $K'(F',G')$ and a Boolean algebra $\B'$ from Lemma~\ref{lemmakfgshrink}. Let $\M'$ be an elementary substructure of $\M$ containing $F'$ and $G'$, whose existence is guaranteed by the downward Löwenheim-Skolem theorem.
    
    The relation $\Vdash$ between elements of $\B'\setminus \{\0\}$ and $L^2$-formulas with parameters from $F'$ and $G'$ defined as
    \[b \Vdash \varphi \iff b\leq \bbl\varphi\bbr'\]
    is a forcing in the sense of Definition~\ref{twosortedforcing}, where $S_0=S_1=\M'$.
\end{theorem}
\begin{proof}
    In this proof the variables $b,c$ and $d$ will only quantify over the poset $P=\B'\setminus\{\0\}$. We will start by proving the following claim.

    \textbf{Claim:} For every $b\in P$, if $b\not \Vdash \varphi$ then $(\exists c\leq b)(c\Vdash \lnot \varphi)$.
    
    \emph{Proof of claim.} In the case that $\bbl\lnot \varphi\bbr'\land b=\0$, we get that $b\leq \bbl\varphi \bbr'$, a contradiction. Otherwise, $b \land \lnot\bbl\varphi\bbr'$ is both under $b$ and under $\bbl\lnot \varphi \bbr '$, thus $c=b\land \bbl\lnot \varphi \bbr'\Vdash \bbl\lnot \varphi\bbr'$ which proves the claim.
    
    In the rest of the proof, we will verify every property a forcing relation has to satisfy.

    \textbf{Extension}: Follows trivially from the transitivity of $\leq$.
    
    \textbf{Stability}: We will assume that for a $b\in P$ and $\varphi$ we have that 
    \[(\forall c\leq b)(\exists d\leq c)(d\Vdash \varphi),\tag{\dag}\]
    where the quantifiers are over $\B'\setminus \{\0\}$. Let $D= \{d\leq b; d\Vdash \varphi\}$. From the properties of suprema it follows that
    \[\Lor_{d\in D} d\leq   \bbl\varphi\bbr',\]
    it remains to show that the supremum $s=\Lor_{d\in D} d$ actually equals $b$. Assume for contradiction that $s<b$, that means that $b\land \lnot s$ is a non-zero element under $b$ and by $(\dag)$ there is some non-zero $d\leq b\land \lnot s$ such that $d\in D$. But then $d$ is both under $s$ and under $b\land \lnot s \leq \lnot s$, which implies that $d=\0$ which is a contradiction.

    \textbf{The negation rule:} We will show that for every $b$: \[b\Vdash \bbl\lnot \varphi\bbr'\text{ if and only if }(\forall c\leq b)(c\not \Vdash \varphi).\]
    The implication from left to right follows by Extension. We will now prove the implication from right to left, assume that $(\forall c\leq b)(c\not \Vdash \varphi)$. By the Claim it follows that $(\forall c\leq b)(\exists d\leq c)(d\Vdash \bbl\lnot \varphi\bbr')$. By Stability, we get that $b\Vdash \bbl\lnot\varphi\bbr'$.

    \textbf{The conjunction rule and the universal quantification rule:} For conjuction it is clear that from $b\leq \bbl\varphi \bbr'$ and $b\leq \bbl\psi\bbr'$ we get that $b\leq \bbl\varphi\bbr'\land \bbl\psi\bbr'$ and since $\bbl-\bbr'$ commutes with conjunctions we get that $b\Vdash \bbl\varphi\land \psi \bbr'$. For universal quantification an analogous argument with infinite suprema holds.

    \textbf{The disjunction rule and the existential quantification rule:} These follow directly from the equations 
    \[\bbl\varphi \lor \psi\bbr' = \bbl\lnot (\varphi \land \psi) \bbr',\]
    and
    \begin{align*}
        \bbl(\exists x)(\varphi(x))\bbr' &= \bbl\lnot (\forall x)(\lnot\varphi(x)) \bbr'\\
        \bbl(\exists X)(\varphi(X))\bbr' &= \bbl\lnot (\forall X)(\lnot\varphi(X)) \bbr'. \qedhere
    \end{align*}
\end{proof}

\subsection{Framework U: Boolean-valued models and the Cohen algebra $\CoRO$}\label{sec:ultrapowerforcing}

Let $M$ and $F$ as defined in Section \ref{F}. In this section, we show that countable $F$-restricted ultrapowers are a forcing construction in the sense of Definition~\ref{forcing}.  

The natural candidate for a complete Boolean algebra is the completion of the countable Boolean algebra $\Pfin$ of definable subsets of the countable index set $I$ modulo the ideal of (standardly) finite sets $\fin$. We denote the completion of this algebra $\CoRO$ (see Section \ref{sec:KAM} for more details).

The details of the interpretation of Framework U are analogous to interpretation of framework K: first a Boolean-valued model is constructed with values of atomic formulas in $\Pfin$, and extended to all formulas with values in the completion $\CoRO$. Then a forcing relation is defined from the Boolean-valued model.

This whole situation can be brought down to a countable size by a L{\"o}wenheim--Skolem theorem, exactly as in Lemma \ref{lemmakfgshrink}. Let $\CoRO'$ be the countable subalgebra of $\CoRO$ obtained by this construction.

\begin{theorem}
     Define a relation between the non-zero elements of $\CoRO'$ and $L$-formulas with parameters $\beta_1, \ldots, \beta_n$ from $F$ by 
    \[p\Vdash \varphi(\beta_1,\dots,\beta_k) \iff p\leq \bbl \varphi(\beta_1,\dots,\beta_k)\bbr.\]
    Then $\Vdash$ is a forcing in the sense of Definition~\ref{forcing} over the set $F$.
\end{theorem}

\begin{proof}
Analogous to the proof of Theorem \ref{theoremkrajicekismuller}.
\end{proof}

Once the restricted ultrapowers are interpreted in Framework AM using the forcing relation $\CoRO'$, one can observe that $F$-restricted ultrapowers $F/G$ via generic ultrafilters $G \sub \CoRO'$ are isomorphic to generic extensions $M[G]$.
    
\begin{theorem}\label{theoremassociateucultrapower}
Suppose $M, I$ and $F$ are as in Section \ref{F}. Then $F/G \cong F[G].$
\end{theorem}
\begin{proof}
The universe of $F[G]$ consists of all closed $L$-terms with parameters from $F$ modulo the equivalence given by $\Th(G)$. Since $F$ is $L$-closed, then for every closed $L$-term $t$ there is $\alpha\in F$ such that $(t=\alpha)\in \Th(G)$, moreover for $\alpha,\beta\in F$, we have that $(\alpha=\beta)\in\Th(G)$ if and only if there is $p\in G$ such that $p\leq \bbl\alpha=\beta\bbr$ which is equivalent to $F/G\models ([\alpha]_\sim=[\beta]_\sim)$, thus there is a bijection between the universes of $F[G]$ and $F/G$. Similarly, for a relational symbol $R\in L$, we have that for $\alpha_1,\dots,\alpha_k\in F$: $R(\alpha_1,\dots,\alpha_k)\in \Th(G)$ if and only if there is $p\in G$ such that $p\leq \bbl R(\alpha_1,\dots,\alpha_k)\bbr$ if and only if $F/G\models R([\alpha_1]_\sim,\dots,[\alpha_k]_\sim)$.
Thus, there is an isomorphism of $F/G$ and $F[G]$ by mapping $[\alpha]_\sim\in F/G$ to the equivalence class of $\alpha$ in $F[G]$.
\end{proof}

\brm
Note that not all ultrafilters on $\CoRO'$ are generic, hence the restricted ultrapowers $F/U$ for ultrafilters $U$ on $\CoRO'$ form a strictly bigger family of structures than generic extensions.\footnote{If $\B$ is a Boolean algebra and $D \sub \B$ is a dense open set, then $\bigvee D = 1_\B$. We call $D$ \emph{unavoidable} if there are finitely many $b_1, \ldots, b_n$ from $D$ such that $b_1 \vee \cdots \vee b_n = 1_\B.$ Otherwise $D$ is called \emph{avoidable}. It is easy to see that every ultrafilter $U$ on $\B$ must meet every unavoidable dense open set $D$. If $D$ is avoidable, then the ideal $I_D$ generated by $D$ is proper (it does not contain $1_\B$). It follows that the dual filter $F_D$ is proper, and can be extended in an ultrafilter $U_D$ which avoids every element of $D$. Finally, note that there always exist avoidable dense open sets which are definable from the model $M$ and a Boolean algebra $\B$, and thus they are included in the definition of genericity in Framework AM (see Definition \ref{forcing algebra}). A canonical example for $\Pomega$ (where $\omega$ denotes natural numbers) is the dense open set of infinite density zero subsets of $\omega$, containing all equivalence classes $[X]$ such that $\lim_{n \to \infty} \frac{\vert{}X \cap n\vert{}}{n}$ equals to $0$.}
\erm

Also note that unless $F$ is full, the truth in the ultrapower $F/G$ is not computed by the Łoś theorem. If $F$ is full, then one can look for a reduction of truth in generic extensions $F/G \cong M[G]$ to Łoś theorem. We will not develop this correspondence in this section but we provide details for the \Krno's forcing in Theorems \ref{th:Rmax} and \ref{th:ultra}.

\brm
We chose to show that restricted ultrapowers are interpretable in framework AM by means of Boolean-valued models to accentuate the analogy with framework K. There is a more direct approach which uses just the algebra $\Pfin$ to define the forcing relation on the atomic formulas and extends it by recursion to all formulas, following Definition \ref{forcing}.  It is easy to see that the resulting forcing relations are the same. This follows by the following more general lemma:

Suppose $\B$ is a complete Boolean algebra and let $\B^+$ be the set of positive elements. Assume that $D \sub \B^+$ is a dense partial order in $\B^+$. 

Define the forcing relation for $p \in \B^+$ in the usual way: $$p \Vdash_{\B} \varphi \iff p \le \bv{\varphi}_\B.$$ Suppose $\Vdash^{\mathrm{At}}_D$ is a relation between elements of $D$ and atomic formulas such that  $\Vdash^{\mathrm{At}}_D$ and $\Vdash_\B$ coincide on the atomic formulas. Let  $\Vdash_D$ be the unique extensions to all formulas, defined by recursion following the recursive definition in Definition \ref{forcing}. Then the two forcing relations are the same:

\begin{lemma}
For all $\varphi$ and all $d \in D$:$$d \Vdash_D \varphi \iff d \Vdash_\B \varphi.$$
\end{lemma}

\begin{proof}
Suppose by induction the equivalence holds for $\varphi$. We show it for $\exists x \varphi$. By the definitions of the respective forcing relations we have: $$d \Vdash_D \exists x \varphi \iff (\forall x \le d, x \in D)(\exists  y  \le x, y \in D)(\exists a \in M) \; y \Vdash_D \varphi(a), $$ and  $$d \Vdash_\B \exists x \varphi \iff d \le \bv{\exists x \varphi} = \bigvee \set{\bv{\varphi(a)}}{a \in M}.$$  By the induction assumption, 

$$d \Vdash_D \exists x \varphi  \iff (\forall x \le d, x \in D)(\exists  y  \le x, y \in D)(\exists a \in M) \; y \le \bv{\varphi(a)}.$$

It follows it suffices to show:

\begin{multline}
(\forall x \le d, x \in D)(\exists  y  \le x, y \in D)(\exists a \in M) \; y \le \bv{\varphi(a)} \iff \\ d \le  \bigvee \set{\bv{\varphi(a)}}{a \in M}.
\end{multline}
The equivalence is true by routine computations over complete Boolean algebras.
\end{proof}
\erm

\section{Generalizations to uncountable structures}\label{sec:4}

In this section, we look at models in Framework K in their original uncountable formulations and interpret their restriction to the number sort as restricted ultrapowers existing in set-theoretic generic extensions of transitive models of set theory. Besides saving the true uncountable character of the models and thus making the computations more natural (as compared with computations reduced to a countable case by a L{\"o}wenheim--Skolem argument), we obtain the additional information that the Boolean algebras corresponding to framework K are isomorphic to well-known random algebras in set theory. It follows that, at the moment, two types of complete Boolean algebras stand out as practically relevant for forcing in arithmetic: the completion of the Cohen forcing $\CoRO$, and the random algebras $\mathcal{B}_\kappa$ generated by the product measure spaces $2^\kappa$. 

The exposition in this section builds on the results in \cite{RH:bounded} by the first author; it adds a new description of Boolean-valued models for random variables as restricted ultrapowers via generic ultrafilters  (compare with Sections \ref{restricted} and \ref{sec:33} where we discussed ultrapowers via ultrafilters on countable sets). We provide a probabilistic interpretation of the forcing arguments which is closer to the methods (and terminology) of \cite{krajicek11} to prove some results on the properties of the linear orders in generic extensions of $\aleph_1$-saturated models of arithmetic. 

\brm
To simplify the exposition from the notational point of view, we will limit our attention to one-sorted theories of bounded arithmetic. This is sufficient for discussing the basic ideas in this section; with more effort the discussion can be extended to two-sorted theories. 
\erm

\subsection{Collapsing Boolean-valued models to two-valued models}\label{sec:41}

Even though Boolean-valued models are sufficient from the point of showing (relative) consistency of a theory $T$, it is sometimes useful to have a $\{0,1\}$-valued model of $T$. Let us fix a language $L$ of arithmetic, for concreteness a language appearing in \Krno's forcing in Section \ref{sec:32}.

A Boolean-valued model $\A$ in the language $L$ can be defined in many ways, but to behave well, it always needs to satisfy some minimal requirements to ensure that all axioms of the predicate logic are valid in $\A$ (i.e., their Boolean value is $1$). A general definition of evaluation of formulas in $\A$ (in the language of set theory) is in Jech \cite[(14.9)]{jech2006}. Once the Boolean evaluation of atomic formulas is defined, the recursive definition in \cite[(14.9)]{jech2006} extends it to all formulas.

\begin{definition}\label{def:boolean-valued}
Suppose $\A$ is a Boolean-valued model over a complete Boolean algebra $\B$. Suppose $U$ is an ultrafilter on $\B$.  Let $[a]_U$ be the equivalence class of $a \in \A$ defined by $[a]_U = \set{b \in \A}{\bv{a=b} \in U}$. The quotient $\AU$ has domain $\set{[a]_U}{a \in  \A}$ with relations $R$ defined by $$R([a_1]_U, \ldots, [a_n]_U]) \iff \bv{R(a_1, \ldots, a_n)} \in U.$$
\end{definition}

The quotient $\AU$ is a $\{0,1\}$-valued structure.

\begin{definition}\label{def:faith}
We say that the quotient $\AU$ is \emph{faithful} if for all sentences $\varphi$ in $L_{\A}$ ($L_{\A}$ is $L$ extended by constants for elements of $\A$):
\begin{equation}\label{U} \bv{\varphi(a_1,\ldots, a_n)} \in U \iff \AU \models \varphi([a_1]_U, \ldots [a_n]_U).\end{equation}
\end{definition}

There are two methods how to collapse $\A$ faithfully: one places requirements on $\A$ and the other on $U$.

\subsubsection{Full models}

\begin{definition}\label{full}
We say that a Boolean-valued model $\A$ is \emph{full} if for every formula $\varphi(x_0,\ldots, x_n)$ with free variables shown and all $a_1, \ldots, a_n \in \A$, there is some $a_0 \in \A$ such that
\begin{equation}\label{eq:full}
\bv{\exists x_0 \varphi(x_0, a_1, \ldots, a_n)} = \bv{\varphi(a_0, \ldots, a_n)}.
\end{equation}
\end{definition}

Since $\bv{\exists x_0 \varphi(x_0, a_1, \ldots, a_n)}$ is by definition equal to $$\bigvee_{a \in \A}\bv{\varphi(a,a_1, \ldots, a_n)},$$ fullness ensures that the supremum is attained by an element in $\A$.

Fullness ensures a faithful transformation of Boolean-valued models:

\begin{theorem}\label{th:U}
Suppose $\B$ is a complete Boolean algebra and $\A$ is a full Boolean-valued model in $\B$. Suppose $U$ is an ultrafilter on $\B$.  Then the quotient structure $\AU$ is a $\{0,1\}$-valued model satisfying for all sentences $\varphi$ in $L_{\A}$:
\begin{equation}\label{U1} \bv{\varphi(a_1,\ldots, a_n)} \in U \iff \AU \models \varphi([a_1]_U, \ldots [a_n]_U).\end{equation}
\end{theorem}

\begin{proof}
One can prove the theorem by induction on the complexity of $\varphi$. The existential quantifier step goes through by fullness:
\begin{equation}
\bv{\exists x \varphi} = \bigvee_{a \in \A}\bv{\varphi(a)} = \bv{\varphi(b)}, \mbox{ for some $b \in \A$},
\end{equation}
hence
\begin{equation}
\bv{\exists x \varphi} \in U \Then \bv{\varphi(b)} \in U \iff \AU \models \varphi([b]_U).
\end{equation}
See Jech \cite[Lemma 14.14]{jech2006} for a complete proof. 
\end{proof}

\textbf{Example 1.} An important example of a full Boolean-valued model is the class $V^{\B}$ of forcing names, where $\B$ a complete Boolean algebra is viewed as a forcing notion (see Jech \cite[Lemma 14.19]{jech2006}). We therefore have the following corollary (where $\in_U$ is the realization of the membership in $V^{\B}/U$ defined by $[a]_U \in_U [b]_U \iff \bv{a \in b} \in U$):

\begin{corollary}
Suppose $V^{\B}$ is a Boolean-valued forcing extension in set theory and $U$ is an ultrafilter on $\B$. Then the quotient $V^{\B}/U$ satisfies $$\bv{\varphi(a_1,\ldots, a_n)} \in U \iff (V^{\B}/U, \in_U) \models \varphi([a_1]_U, \ldots [a_n]_U).$$
\end{corollary}

If $U$ is an ultrafilter in the universe $V$ over some non-atomic Boolean algebra $\B$, then $(V^{\B}/U, \in_U)$ is necessarily an \emph{ill-founded} model of set theory in $V$: if the binary relation $\in_U$ were well-founded, the Mostowski transitive collapse of $(V^{\B}/U, \in_U)$ would be a proper extension of $V$.

\medskip

\textbf{Example 2.} Another example of a full model is the is the Boolean-valued model from Section \ref{sec:32} constructed with the set $\XiM$ of  all random variables (see Theorem \ref{th:Rmax}). As we will discuss in Section \ref{sec:full} below in more detail, models for families of random variables $\Xi$ with restricted complexity are full only for certain classes of bounded formulas (``restricted fullness''). This is essential for obtaining non-trivial independence results because full fullness implies that the respective model is an elementary extension of the model of true arithmetic $\M$ fixed at the beginning (and thus has the same theory); see Theorem \ref{th:ultra}. Compare with Section \ref{restricted} where we discussed these concepts in the countable setting.

\subsubsection{Model-complete and generic ultrafilters}

Suppose $\A$ is a Boolean-valued model with respect to a complete Boolean algebra $\B$ that is not full according to Definition \ref{full}. One can still transform $\A$ faithfully into a $\{0,1\}$-valued model by using an ultrafilter $U$ with some additional properties.

\begin{definition}\label{def:complete}
An ultrafilter $U$ is called \emph{$\A$-complete} if for all formulas $\exists x \varphi(x,a_1, \ldots, a_n)$,
$$\bigvee_{a \in \A} \bv{\varphi(a)} \in U \iff \exists b \in \A,  \bv{\varphi(b)} \in U.$$
\end{definition}

$\A$-complete ultrafilters exist in $V$ if $\A$ is countable: then the number of formulas $\psi$ of the form $\exists x \varphi(x, a_1, \ldots, a_n)$ is also countable, and for each one there is a set $A_{\psi} \sub \A$ such that
\begin{equation}
\label{max}
\bigvee_{a \in \A} \bv{\varphi(a)} = \bigvee_{a \in A_{\psi}} \bv{\varphi(a)}
\end{equation}
 and the family $B_\psi = \set{\bv{\varphi(a)}}{a \in A_{\psi}}$ is a maximal antichain in $\B$. By Rasiowa--Sikorski lemma (equivalently, Baire category theorem), there is an ultrafilter $U \sub B$ which meets every such maximal antichain in $\B$.

\begin{definition}\label{def:generic}
If $\A$ is a forcing model $V^{\B}$, or if $\A$ is analyzed from the point of forcing, an $\A$-complete ultrafilter $U$ is usually called \emph{generic}. We will denote generic ultrafilters by $U_g$.\footnote{It usual to use the letter $G$ but since this might be confused with the set of random functionals in Section \ref{sec:32}, we will use a more descriptive notation $U_g$.}
\end{definition}

If $\A$ is uncountable, a $\{0,1\}$-valued faithful transformation may not exist. However, in applications, the usual set of tricks is used in forcing to have $\A$-complete ultrafilters: most frequently, one starts with a countable transitive model $(V,\in)$ of set theory and defines $\A \in V$ inside $V$. Since $V$ is countable, Rasiowa--Sikorsky lemma ensures the existence an $\A$-complete ultrafilter (in the ambient universe). The application of this method to uncountable Boolean-valued models of arithmetic is described in \cite{RH:bounded}.

\brm
If $U_g$ is a generic ultrafilter for a Boolean-valued model $V^{\B}$, then its completeness ensures that the quotient structure $(V^{\B}/U_g, \in_{U_g})$ is well-founded, and hence it is isomorphic to a transitive model usually denoted $(V[U_g],\in)$. $V[U_g]$ has a more constructive description in which the interpretation $\sigma^{U_g}$ of names $\sigma \in V^{\B}$ is defined by recursion: $\sigma^{U_g}$ is the transitive collapse of the equivalence class $[\sigma]_{U_g}$ (see Jech \cite[Definition 14.27]{jech2006}).
\erm

\subsection{Uncountable structures in Framework K and generic ultrafilters}\label{sec:42}

Suppose $\M$ is a countably saturated model of true arithmetic and $\Omega \in \M$ is a non-standard number. Suppose $\Xi$ is a collection of random variables $\alpha: \Omega \to \M$, $\alpha \in \M$, used to define a Boolean-valued model in Section \ref{sec:32} (the domain is composed of functions in $\Xi$ and the Boolean evaluation of formulas is defined recursively from atomic formulas; see  Definition \ref{def:Boolean}). We will denote the model by $\Km$ using the calligraphic type to emphasize it is a Boolean-valued model (this is the same model as  $K(F,G)$ in Section \ref{sec:32} restricted to the first-sort $F$).

\subsubsection{The maximal model $\Kmax$}\label{sec:full}

In this section we observe that $\Kmax$ is a full Boolean-valued model. While this is not explicit in~\cite{krajicek11} a similar idea appears in the proof of \cite[Lemma 5.5.1]{krajicek11}.

\begin{theorem}\label{th:Rmax}
The Boolean-valued model $\Kmax$ is full.
\end{theorem}

\begin{proof}
To show fullness, we prove by induction on the complexity of formulas that for all $\alpha_1, \ldots, \alpha_n \in \XiM$, and every formula $\varphi$, 

\begin{equation}\label{more}
\bv{\varphi(\alpha_1, \ldots, \alpha_n)} = \set{\omega \in \Omega}{\M \models \varphi(\alpha_1(\omega), \ldots, \alpha_n(\omega))}/\Ic.
\end{equation}

The identity (\ref{more}) will be useful in defining an existential witness for fullness.

Clearly,  (\ref{more}) is true for atomic formulas and for the negation and the conjunction step. We need to verify it holds for the existential step. Let us write only $\alpha$ instead of $\alpha_1, \ldots, \alpha_n$ in the formula $\varphi$ for simplicity of notation. 
By the induction hypothesis, assume 
\begin{equation}\label{ind}
\bv{\varphi(\beta,\alpha)} = \set{\omega \in \Omega}{\M \models \varphi(\beta(\omega), \alpha(\omega))}/\Ic,
\end{equation} for every $\alpha, \beta \in \XiM$.  To show fullness for $\exists x \varphi(x,\alpha)$, we need to find $\gamma \in \XiM$ such that 
$$\bv{\exists x \varphi(x, \alpha)} = \bv{\varphi(\gamma,\alpha)}.$$
By the definition, $$\bv{\exists x \varphi(x, \alpha)} = \bigvee_{\beta \in \XiM}\bv{\varphi(\beta,\alpha)}.$$
The internal set $$S = \set{\omega \in \Omega}{\M\models \exists x \varphi(x, \alpha(\omega))}$$ in $\M$ is modulo $\Ic$ an upper bound of $\set{\bv{\varphi(\beta,\alpha)}}{\beta \in \XiM}$ by (\ref{ind}). Define $\gamma$ internally in $\M$ so that $\gamma(\omega)$ is the $\M$-least integer $m \in \M$ such that $\M\models \varphi(m, \alpha(\omega))$ (if there is one; if not let $m = 0$). Clearly, $$\bv{\varphi(\gamma,\alpha)} = \set{\omega \in \Omega}{\M \models \varphi(\gamma(\omega),\alpha(\omega)}/\Ic = S/\Ic = \bv{\exists x \varphi(x, \alpha)}.$$
This shows that the induction works for $\exists x \varphi(x,\alpha)$, and also ensures fullness for this existential formula.
\end{proof}

If $\Xi$ is not maximal and is restricted by some complexity requirement, the internal witnessing described in the proof of Theorem \ref{th:Rmax} can still be done, but the resulting witness $\gamma \in \XiM$ may fail to have the desired complexity to be in $\Xi$ as well. 

Krajíček \cite{krajicek11} develops methods to determine a \emph{feasible} or \emph{stratified} fullness guaranteeing existential witnesses for specific bounded quantifier classes that match the complexity of $\Xi$. These include decision trees and switching lemmas in which random variables are represented via small rudimentary decision trees or bounded-depth circuits. For formulas $\varphi(x)$ within the targeted syntactic fragment of bounded arithmetic (such as open or $\Sigma^{1,b}_0$-formulas), Krajíček uses the PHP-switching lemmas and reflection principles to prove that the truth-value of $\varphi(x)$ can be partitioned by low-complexity decision trees. Because the partition boundaries and the local witnesses both remain within the restricted complexity class, one can glue them together without jumping out of $\Xi$. In this sense, these models satisfy \emph{restricted fullness}: if $\varphi(x)$ is a formula of appropriate bounded complexity, there exists a random variable $\alpha \in \Xi$ that witnesses $\bv{\exists x \varphi(x)}$, i.e., $\bv{\exists x \varphi(x)} = \bv{\varphi(\alpha)}$.

A faithful transformation from Theorem \ref{th:U} can still be achieved for these restricted models  if a generic ultrafilter $U$ is used.

\subsubsection{A generic ultrapower $\UltMMaxG$ and its submodels}\label{sec:ultrapower}

In order to have faithful transformation as in Theorem \ref{th:U} for all Boolean-valued models in \cite{krajicek11}, we formulated an interpretation of Krajicek's forcing in  \cite{RH:bounded} in a generic extension $V[U_g]$ of a countable transitive model of set theory, where $U_g$ is a generic filter for the complete Boolean algebra $\B$ from Section \ref{sec:32}.  

Identifying $U_g$ with an ultrafilter on $\Omega$ for ease of notation, it is immediate by the definitions that quotient $\UKmaxG$ defined by means of the maximal family of random variable $\XiM$ is actually an ultrapower, which we will denote $\UltMMaxG$, of the starting true model of arithmetic $\M$. All other models $\UKmG$ are its submodels, in fact restricted Boolean ultrapowers $\UltMMG$ themselves according to Definition \ref{def:ultrap}. Note that $\UltMMG$ are not elementary submodels of $\UltMMaxG$ because Łoś theorem does not hold for all formulas as evaluated in the structure $\UltMMG$ (see Remark \ref{rm:complete}).\footnote{It is not difficult to define a two-sorted (``second-order'') generalization of the ultrapower construction. Note that while set theory is one-sorted, the powerset operation makes the theory able to work with many-sorted concepts. In particular, the generic ultrapower construction in set theory over the whole universe $V$ (see \cite[Section 22]{jech2006})  can be seen as the ultimate ultrapower construction for (transfinitely) many-sorted theories, and thus generalizes the ultrapower construction in Section \ref{sec:ultrapower} to more sorts.}

The Boolean ultrapower is a special form of a Boolean quotient caused by the fact that the domain of $\UKmaxG$ is composed of functions with domain $\Omega$, and the Boolean evaluation of atomic formulas is defined by a coordinate-wise validity on $\Omega$. 

This observation highlights the apparent similarity of the number sort of Krajíček's structures (denoted just $K(F)$ in~\cite{krajicek11}) and the restricted ultrapowers, they are essentially the same construction but over a different Boolean algebra.

The following definition is a specific example of restricted ultrapowers already discussed in Section \ref{F}. Note that if $U$ is an ultrafilter on $\B$ we can naturally identify it with an ultrafilter on the internal subsets of $\Omega$ (it will be clear from the context which ultrafilter we use).\footnote{The ultrafilter contains the internal subsets $A$ of $\Omega$ whose equivalence classes $[A]_\Ic$ are in $U$.}

\begin{definition}\label{def:ultrap}
Suppose $\Xi$ is a family of random variables and $U$ is an ultrafilter over $\B$. We say that $\UltMM$ is \emph{an ($\Xi$-restricted) Boolean ultrapower of $M$ by the ultrafilter on $\B$} if:
\bce[(i)]
\item
The domain of $\UltMM$ is composed of equivalence classes $[\alpha]_U$, where $\alpha \in \Xi$ and $[\alpha]_U = [\beta]_U$ for $\alpha,\beta \in \Xi$ if $\set{\omega \in \Omega}{\alpha(\omega) = \beta(\omega)} \in U$.
\item
The interpretation of relations and functions is defined naturally.
\ece
\end{definition}

\brm
Notice that $\UltMMG$ is the model $\Mg$ defined in \cite[Section 5.3]{RH:bounded}.
\erm

We now observe that the fullness of $\Kmax$ proved in Theorem \ref{th:Rmax} readily implies that the Łoś theorem applies to $\UltMMax$, and hence $\UltMMax$ is an elementary extension of the model $\M$. Note that this does not require that $U$ be generic because the model $\Kmax$ is full. Let $U$ be an ultrafilter on the Boolean algebra $\B$ in what follows.

\begin{theorem}\label{th:ultra}
Łoś theorem holds for $\UltMMax$: for every formula $\varphi$ and $\alpha_1,\ldots,\alpha_n \in \XiM$, the following holds: \begin{multline}\UltMMax \models \varphi[[\alpha_1]_U, \ldots, [\alpha_n]_U] \iff \\ 
\bv{\varphi(a_1, \ldots, a_n} \in U \iff
\\ \set{\omega \in \Omega}{\M \models \varphi(\alpha_1(\omega), \ldots, \alpha_n(\omega))} /\Ic \in U.\end{multline}

In particular the canonical embedding $$j: \M \to \UltMMax$$ is elementary.
\end{theorem}

\begin{proof}
The proof is by induction on the complexity of formulas. The existential step is a consequence of fullness of the model and the identity (\ref{more}). \end{proof}

Thus obtaining consistency results in bounded arithmetic via Boolean-valued models in \cite{krajicek11}, it is essential to use families of random variables $\Xi$ which are not closed under the construction of existential witnesses in the \Los theorem. This ensures that $\UKm$ is a submodel of $\UKmax$, but not an elementary one.\footnote{An analogous claim can be made if the Boolean-valued models $\Km$ are not collapsed to $\{0,1\}$-valued models; the difference is just terminological.}

\brm\label{rm:complete}
If the ultrafilter $U$ is not generic (i.e., not $\Km$-complete),  the restricted Boolean ultrapowers $\UKm = \UltMM$ are still submodels of $\UKmax = \UltMMax$, but they may fail to have the property \begin{equation}\label{UR} \bv{\varphi(a_1,\ldots, a_n)} \in U \iff \UKm \models \varphi([a_1]_U, \ldots [a_n]_U).\end{equation} for formulas of higher complexity. The collapse will thus not be faithful for all formulas according to Definition \ref{def:faith}. While unfaithfulness is mostly a cosmetic drawback---because the Boolean values of more complex formulas in Krajíček \cite{krajicek11} are not computed and are not interesting for consistency claims---having $U$ generic over a transitive model $V$ of set theory does provide a more tractable $\{0,1\}$-valued interpretation of all Boolean-valued models $\Km$.
\erm

In the next section we will consider faithful transformations and work with a generic ultrafilter $U_g$.

\subsection{Random extensions of set theory and connections to Framework K}\label{sec:set}

The hard task in forcing in bounded arithmetic is to compute that the inductions scheme holds for a particular class of formulas, while a given principle which holds in true arithmetic is false (such as $\ontoFPHP$ from Section \ref{sec:32}). The methods for achieving this are specific to bounded arithmetic. However, the general setup of forcing allows one to view some arguments in bounded arithmetic from a more general perspective of set-theoretic forcing and discuss model-theoretic properties of the resulting models $\UKmG$.

As we will review in Section \ref{sec:Maharam}, the complete Boolean algebra $\B$ in Section \ref{sec:32} is isomorphic to a well-known algebra used extensively in set-theoretic forcing (the `random' algebra). This means that the models $\UKmG$ exist in standard well-explored models of set theory $V[U_g]$ and can be analyzed there.

In Section \ref{sec:order} we illustrate the analysis by determining some properties of the linear order $\LUKmG$.

\subsubsection{Uniqueness of the probability measure algebra $\B$ and the random integer}\label{sec:Maharam}

As we already mentioned in Section \ref{sec:KAM}, the algebra from framework K is isomorphic to the random algebra $\Bk$ for some uncountable $\kappa$. Let us mention some comparison between the use of $\Bk$ in set theory and in arithmetic. We will focus on $\Bi$. This algebra is well-known in set theory and is used to add $\aleph_1$-many new reals with certain specific properties (quite different from Cohen reals added by the Cohen forcing). \Kr \cite{krajicek11} uses $\Bi$ and a novel way:

\begin{itemize}
\item
In set theory, $\Bi$ is typically used as an iteration of length $\aleph_1$ of the forcing $\Bo$,\footnote{It is known that for every $\omega\alpha < \aleph_1$, $\Bi$ is forcing-equivalent to the iteration $\B_{\omega\alpha} * \dot{\B}_{[\omega\alpha,\aleph_1)}$ (actually, a more general fact is true; but we will not go into details here, see Fremlin \cite[552P]{fremlin5} for more details).} a probability algebra on $2^\omega$ which is isomorphic to forcing Borel subsets of the unit interval on the reals of positive measure. In this interpretation, $\Bi$ adds $\aleph_1$-many new ``random'' reals, with the $\alpha$-th real added by the $\alpha$-th copy of $\Bo$: The generic object for $\Bo$ can be identified with a Cauchy sequence of rationals converging to a new real.

Coincidentally, this is the way Scott \cite{scott67} used $\B_{\omega_2}$ (defined as $\Bi$ on the space $2^{\omega_2}$) to violate the Continuum Hypothesis.
\item
In bounded arithmetic,  the whole algebra $\Bi$ is used to add one new integer. In this reading, adding a new non-standard integer requires a more complex algebra than adding a new standard real.

The claim that $\Bi$ adds a single new integer should be qualified as follows: 
In the notation of Theorem \ref{th:ultra}, $(\M,<)$ can be identified with a suborder of $(\UltMMaxG,<) = \LUKmaxG$, where  integers in $\UKmaxG - \M$ can be viewed as ``generic''. One of these generic integers has the elegant description as the equivalence class of the identify function $\idU$: we call $\idU$ the ``random integer''.\footnote{The name a ``random integer'' is motivated by the fact that $\idU$ defines the whole generic ultrafilter $U_g$ (see Theorem \ref{th:int}), and hence $\Bi$ can be viewed as adding as its generic object the integer $\idU$. Another reason is the analogy with the ``random real" in set theory (which can be identified with $\idU$ as well, where $U_g$ is  now $B_\omega$-generic, essentially by the same argument as in Theorem \ref{th:int}).}

Since the measure on $\M$ does not lead to a complete metric, it is not the case that the generic object for $\Bi$ converges to  $\idU$---the interpretation of the random integer $\idU$ is more complicated: for every $m \in \M$ which is infinitesimally small, the whole interval $(\idU - m, \idU + m)$ is composed of generic integers (see the next Section \ref{sec:order} for more details).
\end{itemize}

One can show that the random integer $\idU$ codes the whole generic filter $U_g$:

\begin{theorem} \label{th:int}
Suppose $V$ is a transitive model of set theory satisfying $\CH$ and $U_g$ is $\Bi$-generic over $V$. Then:
$$V[U_g] = V[\idU],$$
where $V[\idU]$ is the least transitive model of set theory which has the same ordinals as $V$ and contains $\idU$.
\end{theorem}

\begin{proof}
This is Theorem 5.22 in \cite{RH:bounded}, but let us give a self-contained proof using a probabilistic terminology which is more in line with the motivation behind \cite{krajicek11}. The forcing $\Bi$ is equivalently described as forcing with subsets $A \sub \Omega$ of positive measure (Definition \ref{defibooleanalgebra}), and therefore $U_g$ is a generic ultrafilter consisting of certain subsets of $\Omega$. Two random variables $\alpha, \beta$ are identified if the probability (a Boolean-value in $\Bi$) of $\alpha= \beta$ is in $U_g$, i.e., $$\set{\omega \in \Omega}{\alpha(\omega) = \beta(\omega)} \in U_g.$$ Thus if we know $\idU$ and we are given $A \sub \Omega$, we can define a random variable $\alpha^A$ which is the identical function on $A$ and 0 otherwise. Then we can conclude that $A \in U_g$ if and only the probability of $id = \alpha^A$ is in $U_g$; but the last claim is equivalent to $\alpha^A \in \idU$, and hence the whole ultrafilter $U_g$ is definable in $V[\idU]$ as required.
\end{proof}

\subsubsection{Properties of the linear orders $\LUKmG$} \label{sec:order}

To illustrate some properties of the models $\UKmG$, we prove two facts related to the linear order on $\UKmG$. Since $\M$ is saturated, it follows that the linear order is \emph{quasi-$\aleph_1$-dense}: for $m < n$ in $\M$, if the interval $(m,n)=\{i\in \Mc; n< i < m\}$ is externally infinite, it contains externally $\aleph_1$-many integers. If the family $F$ of random integers contains constant functions $\alpha_m: \Omega \to \{m\}$ for every $m\in \M$, the linear order $(\M,<)$ can be identified with a suborder of $\LUKmG$. It is natural to ask whether $\UKmG$ is \emph{quasi-dense} in $\M$, i.e., whether for every $m,n\in M$ with $(m,n)$ externally infinite there is $\alpha \in F$ such that $[\alpha]_{U_g}$ is generic and $m < [\alpha]_{U_g} < n$ in $\LUKmG$, and conversely. 

\begin{theorem}\label{th:q}
With the notation as above:
\bce[(i)]
\item Suppose $m,n \in \M$ and the interval $(m,n)$ is externally infinite. Then there is $\alpha \in \XiM$ with $[\alpha]_{U_g} \not \in \M$ such that $$m < [\alpha]_{U_g} < n$$ in $\LUKmaxG$. It follows that $\UKmaxG$  is quasi-dense in $\M$. 
\item We say that $m \in \M$ is infinitesimally small if $[0,m]$ has measure zero. Then for every infinitesimally small $m \in \M$, the interval $$(\idU - m,\idU + m)$$ is disjoint from $\M$, i.e., it is composed only from generic (new) integers. It follows that $\M$ is not quasi-dense in $\UKmG$ (for any $F$ containing all constant functions).
\ece
\end{theorem}

\begin{proof}
We will give a brief self-contained proof, with a ``probabilistic'' terminology (compare with Theorems 5.25 and 5.28(ii) in \cite{RH:bounded} where the arguments are formulated in the forcing terminology).

\textbf{Regarding (i)}: Let us assume that $(m,n)$ is infinitesimally small (if not, choose a subinterval of $(m,n)$ which is infinitesimally small). Denote $p = |(m,n)|$. Since $p$ is infinitesimally small,  $$|\Omega| = p \cdot l + r,$$ for some non-standard infinitesimally small $l$ and some remainder $r$ (for simplicity assume $r = 0$). 

It follows that $\Omega$ can be internally partitioned into a sequence $\seq{I_i}{i<p}$ of intervals of size $l$. Choose internally a sequence $\seq{m_i}{i<p}$ of distinct integers in the interval $(m,n)$. Let $\alpha$ be a random variable which assigns to every element of $I_i$ the integer $m_i$. Then with probability 1, $[\alpha]_{U_g} \in (m,n)$, yet probability of $\alpha$ being identical to $\bar{m}$ for any $\bar{m} \in (m,n)$ is 0. It follows $\alpha$ is as required. 

\textbf{Regarding (ii)}: Let $m$ be fixed. We want to argue that if $$\aU \in (\idU - m, \idU +m) \iff \set{\omega \in\Omega}{\omega - m < \alpha(\omega) < \omega +m} \in U_g,$$ then the probability of $\alpha$ being identical to any $m^* \in \M$ is $0$. This follows easily: for any fixed $m^*$, the set $\set{\omega \in \Omega}{\omega - m < m^* < \omega +m}$ is contained in the interval $(m^* - m, m^* + m)$ of size $2m$, which is an infinitesimally small number.
\end{proof}

One can check that the $\alpha$ constructed in the proof of (i) also exists in $\UKmPHP$ and can be computed by a depth 1 tree querying where a single pigeon is mapped.

We note that Theorem \ref{th:q}(ii) can be generalized to more random variables (see Theorem 5.28 in \cite{RH:bounded}).

\section{Open questions}\label{sec:open}

\Kr \cite{krajicek11} studies various theories (e.g., open comprehension, $V^0_1$, $Q_2V^0_1$, etc.), verifying the induction/comprehension axioms for a specific class of formulas by variants of quantifier elimination. It would be interesting to see whether there is a definability criterion  for the validity of comprehension in models $K(F,G)$. Since Krajíček's forcing introduces new integers in the number-sort it is not conservative and thus we cannot directly use our inclusion of Krajíček's forcing in the framework of~\AMno.

\begin{problem}
    Does a variant of the Principal theorem (Theorem~\ref{principal theorem}) hold for Krajíček's forcing with random variables?
\end{problem}

We saw in Theorem \ref{pvcomplete} that every countable model of $T_{PV}$ is isomorphic to a restricted ultrapower. Thus all non-provability results for $T_{PV}$ can be obtained by restricted ultrapowers. Is \Krno's forcing similarly complete for theories $V^0_1$ or $V^1_1$? 

\begin{problem}
    Is Krajíček's forcing complete for theories $V^0_1$ or $V^1_1$? 
\end{problem}

The final and the most general problem is the broadening of our knowledge of useful forcing relations in the context of arithmetic, as the current results are quite limited compared to the situation in set theory. It seems that asking for interesting models of stronger theories is currently beyond our reach unless there is a fundamental breakthrough in our understanding. We feel that a more achievable progress is obtaining new unprovability results for $V^0_1$ and its subsystems. 

\begin{problem}
    Can the analogies with set theory, in particular as regards better understanding of uncountable partial orders, be used to find good candidates for new forcing arguments for $V^0_1$ or any of its subsystems?
\end{problem}

\bibliographystyle{plain}

\end{document}